\documentclass[final,1p,times]{elsarticle}

\usepackage{amsmath,amssymb,amsfonts,amsthm,enumerate,multirow,mathtools}
\usepackage{color}
\usepackage{siunitx}
\usepackage[linesnumbered,ruled]{algorithm2e} 
\usepackage{diagbox} 
\usepackage{cleveref}
\usepackage{placeins} 
\usepackage[table]{xcolor}
\usepackage{graphicx}
\usepackage{tikz}
\usepackage{makecell}

\newtheorem{theorem}{Theorem}

 \newcommand{\mrd}{\mathrm{d}}


\usepackage{lineno}

\journal{International Journal of Solids and Structures}

\begin{document}

\begin{frontmatter}

\title{Well-posedness and trivial solutions to inverse eigenstrain problems}

\author[Newc]{CM Wensrich}
\author[Manc]{S Holman}
\author[Manc]{WBL Lionheart}
\author[ANSTO]{V Luzin}
\author[Newc]{D Cuskelly}
\author[ESS]{O Kirstein}
\author[ANSTO]{F Salvemini}

\address[Newc]{School of Engineering, The University of Newcastle, University Drive, Callaghan, NSW 2308, Australia}
\address[Manc]{Department of Mathematics, Alan Turing Building, University of Manchester, Oxford Rd, Manchester M139PL, UK}
\address[ANSTO]{Australian Centre for Neutron Scattering, Australian Nuclear Science and Technology Organisation (ANSTO), Kirrawee NSW 2232, Australia}
\address[ESS]{European Spallation Source ERIC, P.O. Box 176, SE-221 00 Lund, Sweden}

\begin{abstract}
We examine the well-posedness of inverse eigenstrain problems for residual stress analysis from the perspective of the non-uniqueness of solutions, structure of the corresponding null space and associated orthogonal range-null decompositions.  Through this process we highlight the existence of a trivial solution to all inverse eigenstrain problems, with all other solutions differing from this trivial version by an unobservable null component.  From one perspective, this implies that no new information can be gained though eigenstrain analysis, however we also highlight the utility of the eigenstrain framework for enforcing equilibrium while estimating residual stress from incomplete experimental data.  Two examples based on measured experimental data are given; one axisymmetric system involving ancient Roman medical tools, and one more-general system involving an additively manufactured Inconel sample.  We conclude by drawing a link between eigenstrain and reconstruction formulas related to strain tomography based on the Longitudinal Ray Transform (LRT).  Through this link, we establish a potential means for tomographic reconstruction of residual stress from LRT measurements.
\end{abstract}

\begin{keyword}
Eigenstrain; Inherent strain, Residual stress;
Additive manufacture; Strain tomography
\end{keyword}

\end{frontmatter}

\section{Introduction}

Residual stresses are ubiquitous in both nature and the manufactured world.  In engineering in particular, they play a key role in determining the mechanical performance of machine components. Of particular note, residual stress fields have a dramatic effect on fatigue performance, with compressive stresses known to inhibit the propagation of fatigue cracks \cite{Fatigue:james2007residual,Fatigue:kohler2012residual,Fatigue:webster2001residual}; a fact that is often exploited by various surface treatment processes.

Practically speaking, all manufacturing processes create or alter residual stresses to some degree, and the various mechanisms through which this occurs have been the focus of significant research effort.  This effort is from two perspectives; 
\begin{enumerate}
    \item{Modeling and predicting the residual stress that is created by various manufacturing processes.}
    \item{Measurement and characterisation of residual stress fields within components.}
\end{enumerate}

From both of these perspectives, the \emph{eigenstrain} concept has become a key theoretical framework \cite{Eig:cao2002inverse,Eig:Hill1996Determination,Eig:jun2010eval,Eig:korsunsky2007variational,Eig:korsunsky2008eigen,Eig:korsunsky2017teaching,Eig:mura2013micromechanics,Eig:uzun2023voxel,Eig:cai2022novel}.  Closely related to the \emph{inherent strain} concept for analysing welds and now additive manufacturing (e.g. \cite{Inh:bugatti2018limitations,Inh:setien2019empirical}), the eigenstrain approach views residual stress as being the elastic response to an underlying inelastic strain within the material.  Pioneered by Mura and others (e.g. \cite{EigMura:lin1973elastic,EigMura:mura1976elastic}), the concept was initially focused on understanding and homogenising the effect of inclusions as an extension of Eschelby theory \cite{eshelby1957determination}, however, it has since developed into a common approach for modeling and understanding macroscopic residual stress distributions. Both applications can still be found in current literature (e.g. \cite{richeton2024energy,ma2014principle,yan2022analytical}, verses \cite{Eig:cai2022novel,Eig:jun2010eval,Eig:uzun2023voxel}).

In this paper we provide a brief overview of the eigenstrain framework before examining the mathematical foundations of the approach from the perspective of an inverse problem.  In particular, the ill-posed nature of the problem is highlighted and discussed in terms of the interpretation of solutions.  We conclude by drawing a key link between eigenstrain analysis and an open problem in the area of strain tomography from energy-resolved neutron imaging.

\section{Governing equations and an overview of eigenstrain analysis}

Consider a physical object represented by the finite domain $\Omega \subset \mathbb{R}^3$ with Lipschitz boundary and a unit outward-normal vector $n$ on $\partial \Omega$.  We assume $\Omega$ exists in the absence of body forces (e.g. gravity, magnetic/electric forces) and is static (i.e. all derivatives with respect to time are zero).  We also assume that no external forces are applied to the surface; i.e.
\begin{equation}
    \sigma \cdot n = 0 \text{ on } \partial \Omega,
    \label{NoTraction}
\end{equation}
where $\sigma \in L^2(\mathcal{S}^2,\Omega)$ is the rank-2 stress tensor on $\Omega$\footnote{Throughout the paper we will make reference to the following spaces;
\begin{itemize}
    \item[--]{$H^l(\mathcal{S}^m;\Omega)$: The Sobolev space of symmetric tensor fields of rank-$m$ on $\Omega$ with square-integrable $l^{th}$ derivatives.  For rank-0 scalar fields with square-integrable $l^{th}$ derivatives on $\Omega$ we write $H^l(\Omega)$.}
    \item[--]{$L^2(\mathcal{S}^m;\Omega)$: The Hilbert space of symmetric tensor fields of rank-$m$ on $\Omega$.  $H^0$ is equivalent to $L^2$.  For rank-0 scalar fields on $\Omega$ we write $L^2(\Omega)$.}
\end{itemize}}.  We refer to this condition as the \emph{no traction} constraint or $\sigma$ having \emph{zero boundary flux}.

\begin{figure}[!htb]
    	\centering
        \includegraphics[width=0.3\linewidth]{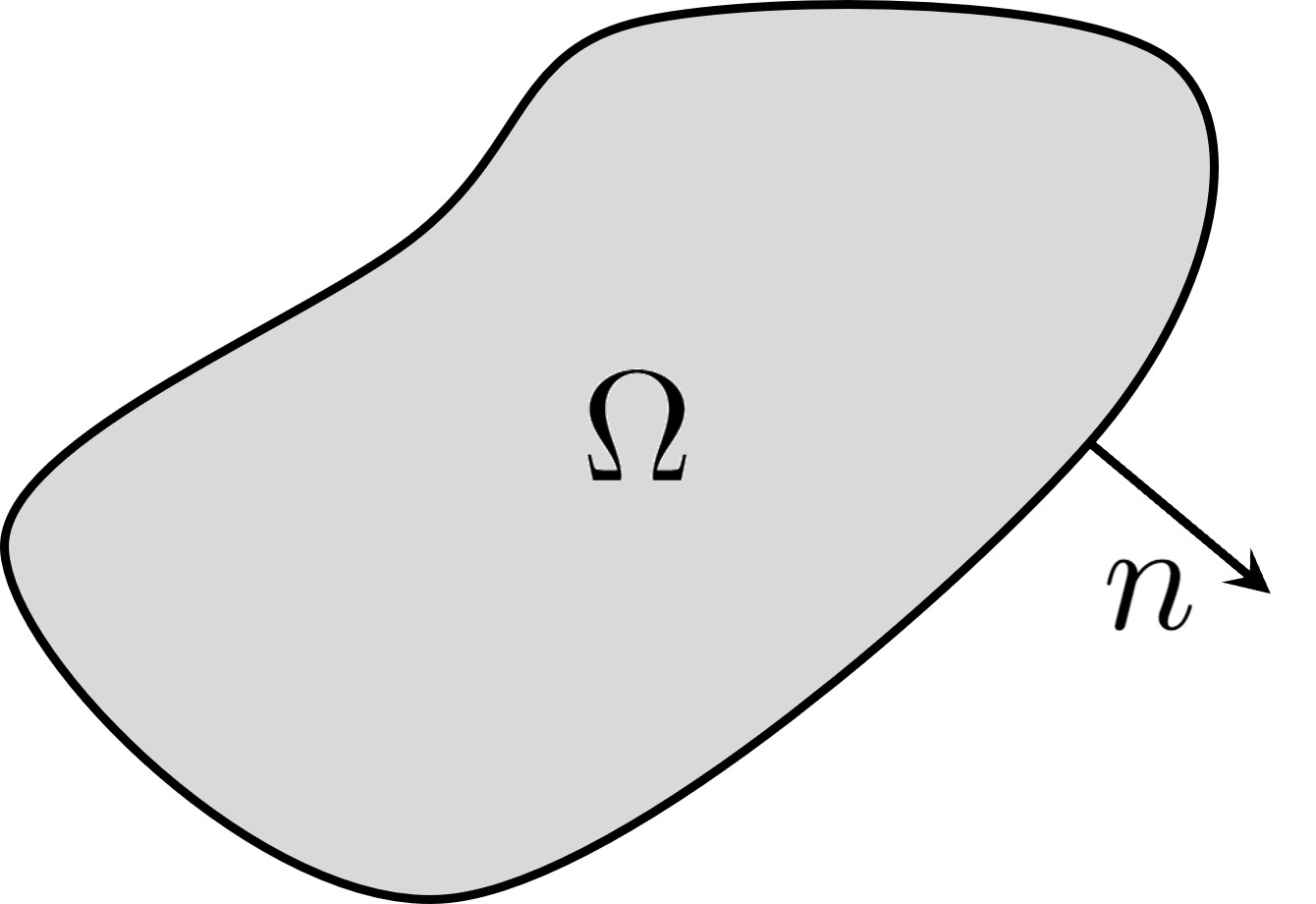}
    	\caption{An object represented by the domain $\Omega$ with surface normal $n$.}
\end{figure}

Total strain at any point within this object originates from a displacement field and can be written
\begin{equation}
    \label{StrainDisp}
    \epsilon^T=\nabla_sU^T
\end{equation}
where $U^T \in H^1(\mathcal{S}^1;\Omega)$, and $\nabla_s$ is the symmetric gradient operator with components 
\[
[\nabla_s U]_{ij}=\frac{1}{2}\Big(\frac{\partial U_i}{\partial x_j}+\frac{\partial U_j}{\partial x_i}\Big).
\]

In general, this total strain can be decomposed as
\begin{equation}
    \label{StrainDecomp}
    \epsilon^T=\epsilon+\epsilon^*
\end{equation}
where $\epsilon$ is the elastic component of strain, and $\epsilon^*$ is an arbitrary eigenstrain representing inelastic strain generated through plasticity, phase change, thermal expansion, etc.  Assuming linear elasticity, $\epsilon$ is related to $\sigma$ though Hooke's law, which for a general anisotropic material takes the form
\[
    \sigma = C:\epsilon,
\]
with $C$ representing the usual rank-4 tensor of elastic constants and $:$ representing contraction over two indices (i.e. $\sigma_{ij}=C_{ijkl}\epsilon_{kl}$).  In this context, we consider the tensor $C$ to represent a bijection from elastic strain to stress with an inverse represented by the usual rank-4 compliance tensor $S$.  In the case of an isotropic material, all components of $C$ can be specified by 2 elastic properties.  For example, in terms of Young's modulus $E$ and Poisson's ratio $\nu$,
\[
    C_{ijkl}= \frac{E}{1+\nu} \Big(\frac{\nu}{1-2\nu}\delta_{ij} \delta_{kl} + \frac{1}{2}\left( \delta_{ik} \delta_{jl} + \delta_{il} \delta_{jk}\right)\Big).
\]

The phrase \emph{eigenstrain analysis} generally refers to one of two problems;
\begin{enumerate}
    \item{\emph{The forward problem:} Given a prescribed eigenstrain, we would like to determine the stress field that it creates.}
    \item{\emph{The inverse problem:} Given an observed residual stress field, determine the eigenstrain field that created it.}
\end{enumerate}

In this endeavour, the relationship between residual stress and eigenstrain can be formed through an assumption of mechanical equilibrium at every point within $\Omega$.  In the absence of body forces (e.g. gravity), this holds that
\begin{equation}
\label{Equilibrium}
    \text{Div }\sigma = 0.
\end{equation}
where $\text{Div}$ is the divergence operator on symmetric rank-2 tensor fields. 
It follows from \eqref{NoTraction} and the Gauss-Ostrogradsky divergence theorem that the average stress within $\Omega$ is zero since
\[
\bar\sigma=\frac{1}{V_\Omega}\int_\Omega \sigma \,\mathrm{d}V =\frac{1}{V_\Omega}\int_{\partial\Omega} (\sigma \cdot n) \otimes x \,  \mathrm{d}A =0,
\]
where $V_\Omega$ is the volume of $\Omega$, $\otimes$ is the dyadic (outer) product of vectors and $\mathrm{d}V$ and $\mathrm{d}A$ are the volume and surface area measures on $\Omega$ and $\partial \Omega$ respectively.  This condition is often referred to as the `stress-balance' when analysing measured residual stress fields.

Combining \eqref{Equilibrium} with \eqref{StrainDisp}, \eqref{StrainDecomp} and Hooke's law, the governing equations for $U^T$ can be written
\begin{equation}
    \label{EquilibriumEigenstrain}
    \text{Div }\big(C:\nabla_s U^T\big) = \text{Div }\big(C:\epsilon^*\big),
\end{equation}
subject to the boundary condition
\begin{equation}
    \label{BoundaryCondition}
    \big(C: \nabla_s U^T\big) \cdot n \Big|_{\partial\Omega}= \big(C:\epsilon^*\big) \cdot n \Big|_{\partial\Omega}.
\end{equation}

A variety of tools exist for solving this forward problem.  In general, this is usually achieved through the finite element method such as that implemented in commercial `structural analysis' packages, or PDE solvers such as the MATLAB PDE Toolbox \cite{MATLAB_PDEtoolbox} and the FEniCS open-source package \cite{alnaes2015fenics,logg2012automated}.  Often the rigidity of commercial packages limits the analysis to isotropic eigenstrain implemented through manipulation of thermal expansion analysis capability, however in the broader context of this approach, the right-hand-sides of \eqref{EquilibriumEigenstrain} and \eqref{BoundaryCondition} can be interpreted as a spatially varying body force and apparent surface traction respectively.

\section{The null space of eigenstrain fields}

We turn our attention to the inverse problem and begin by recognising that a large set of null eigenstrain fields exist.  These are non-zero eigenstrains that result in no stress within the body and represent a significant issue in terms of the uniqueness of possible solutions.  While our language is not entirely precise, we refer to these null fields as being in the null space of the forward eigenstrain problem (i.e. the forward problem considered to be a mapping from eigenstrain to residual stress).

These null fields are characterised by the following theorem;

\begin{theorem}
    \label{NullFields}
    Any eigenstrain of the form $\nabla_s U^*$ for some vector potential $U^*\in H^1(\mathcal{S}^1;\Omega)$ is in the null space of the forward eigenstrain problem, and conversely, any null eigenstrain field can be written in the form $\epsilon^*=\nabla_s U^*$.
\end{theorem}

\begin{proof}[Proof of Theorem \ref{NullFields}]
    
    We assume $\epsilon^*= \nabla_s U^*$ and denote the resulting displacement field in $\Omega$ as $U^T$.  In this case the governing equations \eqref{EquilibriumEigenstrain} become
    \[
        \text{Div }\big(C:\nabla_s (U^T-U^*)\big) =0
    \]
    with the boundary condition \eqref{BoundaryCondition}
    \[
    \big(C:\nabla_s (U^T-U^*)\big)\cdot n \Big|_{\partial \Omega} =0.
    \]
    
    This elliptic boundary value problem has the unique solution $\nabla_s(U^T-U^*)=0$ which implies $U^T$ and $U^*$ are identical up to any infinitesimal rigid body motion in the kernel of $\nabla_s$.
    
    From \eqref{StrainDecomp}, we see that $\epsilon=\nabla_s (U^T-U^*)$, and hence $\epsilon^*=\nabla_s U^*$ implies elastic strain and stress (through Hooke's law) are both zero.  Trivially, we also observe from \eqref{StrainDecomp} that $\sigma=0$ implies that $\epsilon^*=\nabla_s U^T$.
\end{proof}

In the language of applied mechanics, this implies that the \emph{incompatibility} of any given eigenstrain field is responsible for generating residual stress.  Alternatively, we can say that null eigenstrains satisfy the compatibility relation which can be expressed as a vanishing Saint-Venant operator
\[
[W\epsilon^*]_{ijkl}=\frac{\partial^2\epsilon^*_{ij}}{\partial x_k\partial x_l}  + \frac{\partial^2\epsilon^*_{kl}}{\partial x_i\partial x_j}  -\frac{\partial^2\epsilon^*_{il}}{\partial x_j\partial x_k} - \frac{\partial^2\epsilon^*_{jk}}{\partial x_i\partial x_l}=0,
\]
which in $\mathbb{R}^3$ can be simplified to six unique components as specified by the 2-rank symmetric incompatibility tensor
\begin{equation}
    \mathcal{R}\epsilon^*=\nabla \times (\nabla \times \epsilon^*)^T=0.
\end{equation}

Practically this means the inverse problem is inherently ill-posed; any reconstruction of eigenstrain from a known residual stress is not unique.  Without due care, a given numerical solution to the inverse problem  will contain an arbitrary component from the null space chosen on the basis of some regularisation of the problem.  This regularisation may be explicitly stated in terms of a set of assumptions imposed on the system (e.g. a particular form of the stress or strain tensor), or implicitly applied by the numerical optimisation routine (e.g. a solution that minimises the $L^2$-norm via a Moore-Penrose pseudo-inverse).  The former may have some physical basis, the latter is not ideal.  Hill \cite{Eig:Hill1996Determination} addresses this issue by stating ``eigenstrain fields that result in the same residual stress distribution are considered to be equivalent" however the failure to characterise such equivalent fields makes any direct comparison fraught.

Note that the actual physical eigenstrain applied to a sample may contain a component from the null space.  The corollary of Theorem \ref{NullFields} is that this component can never be recovered from measured stress data.

Practically, we want a consistent approach that chooses the same or a predictable member of the null space each time.  Ideally this is achieved by selecting the unique member of the orthogonal complement to the null space that corresponds to the given stress field.  This is the focus of the following section.

\section{Range-null decomposition of eigenstrain fields}
\label{Sec:Range-NullDecomp}

Consider the following operator inner product of rank-2 tensors $a,b\in L^2(\mathcal{S}^2;\Omega)$,
\begin{equation}
\label{InnerProd}
\langle a,b\rangle _C = \int_\Omega (C:a):b \quad \mathrm{d}V.
\end{equation}
Note that since $C$ is positive semi-definite, this meets the requirements of being a valid inner product.  If acting on an elastic strain, the norm associated with this inner product is twice the strain energy within $\Omega$.

We now use this inner product to form an orthogonal decomposition of strain tensors.  We start with the null space of the forward eigenstrain problem;
\[
W=\{\epsilon \in L^2(\mathcal{S}^2;\Omega) : \epsilon = \nabla_s U \text{ for some } U\in H^1(\mathcal{S}^1;\Omega) \}.
\]

In terms of \eqref{InnerProd}, any eigenstrain field orthogonal to this space satisfies
\[
\langle \epsilon^{\perp_C} , \nabla_s U \rangle_C = 0 ,\quad \forall U\in H^1(\mathcal{S}^1;\Omega)
\]
which, through integration by parts, can be written,
\begin{align*}
0&= \int_\Omega (C:\epsilon^{\perp_C}):\nabla_s U \quad \mathrm{d}V \\
&= \int_{\partial\Omega} [(C:\epsilon^{\perp_C})\cdot n] \cdot U \quad \mathrm{d}A - \int_\Omega \mathrm{Div}(C:\epsilon^{\perp_C}) \cdot U \quad \mathrm{d}V.
\end{align*}

Given that this is true for arbitrary $U$, it implies that the orthogonal complement to $W$ is of the form
\[
W^C=\{\epsilon^{\perp_C} \in L^2(\mathcal{S}^2;\Omega):\mathrm{Div}(C:\epsilon^{\perp_C}) =0 \text{ and }
(C:\epsilon^{\perp_C})\cdot n|_{\partial\Omega} =0\}.
\]
leading to the orthogonal decomposition
\[
L^2(\mathcal{S}^2;\Omega)=W \oplus W^C,
\]
whereby any eigenstrain field can be decomposed into a null component and its complement of the form
\begin{equation}
\label{Decomp}
\epsilon^*=\nabla_s U^*+\epsilon^{*\perp_C}.
\end{equation}

We now seek a solution to the inverse eigenstrain problem associated with a given stress field $\sigma$ in terms of this decomposition.  From this perspective, \eqref{EquilibriumEigenstrain} and \eqref{BoundaryCondition} become
\begin{align*}
\mathrm{Div}(C:\nabla_sU^T)&=\mathrm{Div}(C:\nabla_s U^*) \text{, and }\\
(C:\nabla_sU^T)\cdot n \Big|_{\partial\Omega}&=(C:\nabla_s U^*)\cdot n \Big|_{\partial\Omega},
\end{align*}
which has a family of solutions of the form $dU^T=dU^*$.  Hence from \eqref{StrainDecomp} and \eqref{Decomp} we see that $\epsilon^{*\perp_C}=-\epsilon$ is a solution to the original eigenstrain problem.

This is quite a trivial solution, and perhaps somewhat unsatisfying.  It implies any given inverse eigenstrain problem can be solved trivially through Hooke's law by an eigenstrain of the form
\[
\epsilon^{*\perp_C}=-S:\sigma,
\]
where $S$ is the usual rank-4 tensor of elastic properties (i.e. the inverse of $C$).  Any other solution differs from this trivial one by some unobservable potential of the form $\nabla_sU^*$.

None-the-less, there are still grounds to continue as a typical eigenstrain problem involves limited knowledge of the stress field; usually at a finite set of points and often incomplete in terms of known components.  In this context, the eigenstrain framework provides one possible mechanism through which stress fields satisfying \eqref{NoTraction} and \eqref{Equilibrium} can be generated and fitted to data.  In this endeavor, orthogonal decompositions such as that above can be used to dramatically reduce the search space and bring consistency (i.e. uniqueness) to the final result.  We now demonstrate this process - beginning with the special case of axisymmetric systems.

\section{Special case: Axisymmetric systems}

\subsection{Governing equations}

Consider a long cylindrical object of radius $R$ with no load applied to the outer surface. The radial coordinate $r$ measures out from the centreline and the $z$-coordinate is aligned with the axis.  The object is assumed to be subject to axisymmetric eigenstrains in the three coordinate directions
\[
\epsilon_{rr}^*(r), \epsilon_{\theta\theta}^*(r), \epsilon_{zz}^*(r) \quad \text{for }r \in [0,R],
\]
with all other components assumed to be zero.

In response the material deforms according to a radial $U_r(r)$ and axial displacement $U_z(r)$  from which we define a total strain of
\begin{align}
    \label{StrainDisp1}
    \epsilon_{rr}^T &=\epsilon_{rr}+\epsilon_{rr}^*=\frac{\mrd U_r}{\mrd r}\\
    \label{StrainDisp2}
    \epsilon_{\theta\theta}^T &=\epsilon_{\theta\theta}+\epsilon_{\theta\theta}^*=\frac{U_r}{r}\\
    \label{StrainDisp3}
    \epsilon_{zz}^T &=\epsilon_{zz}+\epsilon_{zz}^*=\frac{\mrd U_z}{\mrd z},
\end{align}
and $\epsilon^T_{r\theta}=\epsilon^T_{rz}=\epsilon^T_{\theta z}=0$.

In the infinite case $U_z=0$ and the total strain in the $z$-direction vanishes.  For a `sufficiently long' cylinder, we can assume the total strain in the axial direction is constant and equal to $\epsilon_{zz}^T=\bar\epsilon_{zz}$.

Far from the end, the axisymmetric and `long' assumptions force all shear stress components to zero along with all stress gradients in the $z$-direction.  From this, the governing equations in terms of stress \eqref{Equilibrium} become
\begin{equation}
    \label{Equ}
    \frac{\mrd \sigma_{rr}}{\mrd r}+\frac{\sigma_{rr}-\sigma_{\theta\theta}}{r}=0.
\end{equation}
Note that $\sigma_{zz}$ has no constraint other than a stress balance of the form
\begin{equation}
    \label{NetZeroForce}
    \int_0^Rr\sigma_{zz}\mathrm{d}r=0.
\end{equation}

From Hooke's law we can write
\begin{align}
    \label{HookesLaw1}
    \sigma_{rr}&=\frac{E}{(1+\nu)(1-2\nu)}\big[ (1-\nu)\epsilon_{rr}+\nu\epsilon_{\theta\theta}+\nu\epsilon_{zz}\big]\\
    \label{HookesLaw2}
    \sigma_{\theta\theta}&=\frac{E}{(1+\nu)(1-2\nu)}\big[ \nu\epsilon_{rr}+(1-\nu)\epsilon_{\theta\theta}+\nu\epsilon_{zz}\big]\\
    \label{HookesLaw3}
    \sigma_{zz}&=\frac{E}{(1+\nu)(1-2\nu)}\big[ \nu\epsilon_{rr}+\nu\epsilon_{\theta\theta}+(1-\nu)\epsilon_{zz}\big],
\end{align}
and combine this with the strain-displacement relationship and \eqref{Equ} to write the governing equation for $U_r$ as
\begin{equation}
    \label{GovEq}
    \frac{\mrd^2U_r}{\mrd r^2}+\frac{1}{r}\frac{\mrd U_r}{\mrd r}-\frac{U_r}{r^2}=b,
\end{equation}
where
\[
    b=\frac{\mrd \epsilon_{rr}^*}{\mrd r}+
    \frac{\nu}{1-\nu}\Bigg(\frac{\mrd\epsilon_{\theta\theta}^*}{\mrd r}+\frac{\mrd \epsilon_{zz}^*}{\mrd r}\Bigg)+
    \frac{1-2\nu}{1-\nu}\frac{\epsilon_{rr}^*-\epsilon_{\theta\theta}^*}{r}.
\]

Boundary conditions for $U_r$ stem from continuity at the centre and the absence of radial stress on the surface (i.e. $\sigma_{rr}(R)=0$) and can be written in the form
\begin{align}
    \label{Continuity}
    U_r(0)&=0, \text{ and}\\
    \label{ZeroTraction}
    \Big[\frac{\mrd U_r}{\mrd r}+\frac{U_r}{r}\Big]_{r=R}&=\Big[(1-\nu)\epsilon_{rr}^*+\nu\epsilon_{\theta\theta}^*+\nu(\epsilon_{zz}^*-\bar{\epsilon}_{zz})\Big]_{r=R}
\end{align}
These relations form the governing equations and boundary conditions for $U$ (and hence $\sigma$) corresponding to a given eigenstrain. 

\subsection{Null axisymmetric fields}

As before, we can characterise the null space of this system as the gradients of potential fields.  In this axisymmetric system, these take the form
\begin{align*}
    \epsilon_{rr}^*&=\frac{\mrd U}{\mrd r}\\
    \epsilon_{\theta\theta}^*&=\frac{U}{r}\\
    \epsilon_{zz}^*&=\bar\epsilon_{zz}
\end{align*}
for any $U(r)$ such that $U(0)=0$ and constant $\bar\epsilon_{zz}$.  Alternatively, this can be expressed as the condition
\begin{align}
    \epsilon_{rr}^*&=\frac{\mrd }{\mrd r} \big(r\epsilon_{\theta\theta}^*\big)\\
    \epsilon_{zz}^*&=\bar\epsilon_{zz}
\end{align}
for any constant axial strain $\bar\epsilon_{zz}$.

Using a similar process to before, we can construct the orthogonal complement to this null space with respect to any given inner product.  In this case we use the usual inner product for $L^2(\mathcal{S}^2;\mathbb{R}^3)$ which can written for our axisymmetric system as
\[
\langle a,b \rangle = 2\pi \int_0^R r (a:b) \mathrm{d}r
\]

In terms of this inner product, any eigenstrain field orthogonal to the null space is such that
\[
\int_0^R r \Big( \epsilon^{*\perp}_{rr}\frac{\mrd U}{\mrd r}+\epsilon^{*\perp}_{\theta\theta}\frac{U}{r}+\epsilon^{*\perp}_{zz}\bar\epsilon_{zz}\Big) \mathrm{d}r = 0,
\]
$\forall U\in H^1([0,R])$ and $\bar\epsilon_{zz} \in \mathbb{R}$.  Given the arbitrary nature of $U$ and $\bar\epsilon_{zz}$, this implies
\[
\int_0^R r\epsilon^{*\perp}_{zz} \mathrm{d}r = 0 \quad \text{and} \quad \int_0^R r \Big( \epsilon^{*\perp}_{rr}\frac{\mrd U}{\mrd r}+\epsilon^{*\perp}_{\theta\theta}\frac{U}{r} \Big) \mathrm{d}r =0.
\]
Integrating by parts, the second expression becomes
\[
r\epsilon^{*\perp}_{rr}U\Big|_{r=R} + \int_0^R \Big(\epsilon^{*\perp}_{\theta\theta}-\frac{\mrd}{\mrd r}(r\epsilon^{*\perp}_{rr})\Big)U \mathrm{d}r =0.
\]
and once again, the arbitrary nature of $U$ implies
\[
\epsilon^{*\perp}_{rr}(R)=0 \quad \text{and} \quad
\epsilon^{*\perp}_{\theta\theta}=\frac{\mrd}{\mrd r}(r\epsilon^{*\perp}_{rr}).
\]

So in terms of the inverse eigenstrain problem, we can restrict our search to eigenstrain fields that satisfy these conditions without any loss of generality.

\subsection{Polynomial representation}

We now consider the case where the axisymmetric eigenstrain components are expressed as polynomials of order $l-1$ of the form 
\begin{equation}
    \label{PolyEigenstrain}
    \epsilon_{rr}^*=\sum_{i=1}^l{f_ir^{l-i}}, \quad \epsilon_{\theta\theta}^*=\sum_{i=1}^l{g_ir^{l-i}} \quad \text{and} \quad\epsilon_{zz}^*=\sum_{i=1}^l{h_ir^{l-i}}
\end{equation}
From this it follows that the right hand side of \eqref{GovEq} is an order $l-2$ polynomial of the form;
\[
    b=\sum_{i=1}^{l-1}{b_ir^{l-1-i}},
\]
with the coefficients $\{b_i\}$ constructed from $\{f_i\}$, $\{g_i\}$ and $\{h_i\}$.

The forward problem consists of finding the corresponding solution to \eqref{GovEq} with respect to this driving term.  In this regard, the homogeneous part can be written
\[
U_0=\alpha r+\frac{\beta}{r},
\]
and from \eqref{Continuity} we see that $\beta=0$.  Solving for a particular solution, we construct the general solution as a polynomial of order $l$ of the form
\begin{equation}
    U_r=\alpha r + U_p,
\end{equation}
where
\[
U_p=\sum_{i=1}^{l+1}{\alpha_ir^{l+1-i}}
\]
and
\[
\alpha_i=\begin{cases}
            0 & \text{ for } i=l \text{ or } l+1,\\
            \frac{b_i}{(l-i)^2-1} & \text{ otherwise.}
        \end{cases}
\]

The zero-traction boundary condition \eqref{ZeroTraction} then provides
\[
    \alpha+\nu\bar{\epsilon}_{zz}=\Big[
    (1-\nu)\big(\epsilon_{rr}^*-\frac{\mrd U_p}{\mrd r}\big)
    +\nu\big(\epsilon_{\theta\theta}^*-\frac{U_p}{r}\big)
    +\nu\epsilon_{zz}^*
    \Big]_{r=R},
\]
and similarly, net-zero axial force \eqref{NetZeroForce} holds that
\[
    \label{Eq2}
    2\alpha+\frac{1-\nu}{\nu}\bar{\epsilon}_{zz}=\frac{2}{R^2}\int_0^R{
    r\Big[
    \big(\epsilon_{rr}^*-\frac{\mrd U_p}{\mrd r}\big)
    +\big(\epsilon_{\theta\theta}^*-\frac{U_p}{r}\big)
    +\frac{1-\nu}{\nu}\epsilon_{zz}^*
    \Big]
    \mathrm{d}r}.
\]
Together these can be solved for the unknown constants $\alpha$ and $\bar{\epsilon}_{zz}$.

From this solution for $U_r$, we can compute elastic strain from the earlier strain-displacement relationships \eqref{StrainDisp1}-\eqref{StrainDisp3}, and hence stress through Hooke's law \eqref{HookesLaw1}-\eqref{HookesLaw3}.  The inverse problem is then the task of determining the coefficients $\{f_i\}$, $\{g_i\}$ and $\{h_i\}$ for $i=1\ldots l$ that correspond to a given stress distribution.  

In this context, we can state the condition for a null eigenstrain field as
\[
    f_i=\Big(l-i+2-\frac{1}{l-i}\Big)g_i \quad \text{and} \quad h_i=0 \text{ for } i \ne l
\]
Similarly, for an eigenstrain in the orthogonal compliment 
\begin{equation}
    \label{PolyConstraints}
    g_i=(l-i+1)f_i, \quad f_l=-\sum_{i=1}^{l-1} f_i R^{l-i} \quad \text{and} \quad h_l=-2 \sum_{i=1}^{l-1} \frac{h_i}{l-i+2}R^{l-i}
\end{equation}

From the perspective of the inverse problem as a least-squares minimisation problem, we can use the latter to dramatically reduce the search space.  We now demonstrate this process in the context of measured experimental data from a neutron scattering experiment.

\section{Example 1: Residual stress in ancient Roman artifacts}

\subsection{Context, samples and experimental measurements}

Figure \ref{romantools} shows four ancient Roman medical tools attributed to the 1$^{st}$ and 2$^{nd}$ century CE from the collection of the RD Milns Antiquities Museum at the University of Queensland\footnote{Acquisition numbers 15.001 d (i) to (iv)}.  These four tools are examples of spathomele or spatula probes found within a larger set of instruments including scalpels, hooks, forceps and a bone-lever, along with a medical box and instrument case.  The four probes each consist of a long shaft with an olivary point at one end and a spatula at the other. They are considered to be pharmaceutical tools with the olivary end used for stirring medicaments and the spatula for spreading them on the affected body part.  The collection was also found with some surviving medicine or ointment within a separate case.  

They were all purchased by the Museum in 2015 with funds from Dr Glenda Powell AM in Memory of Dr Owen Powell OAM.  Prior to this, they were in a French private collection which in turn purchased them from another private collection in Japan.

As a part of an experiment aimed at establishing the original manufacturing techniques used in their production, residual stress within the long cylindrical section of each probe was measured using the KOWARI engineering diffractometer within the Australian Centre for Neutron
Scattering (ACNS) at the Australian Nuclear Science and Technology Organisation (ANSTO)
\cite{kirstein2009strain}.  It was hypothesised that the eigenstrain distribution and resulting residual stress field embedded during manufacture may help to rule-out or support potential production techniques.

Within the long cylindrical sections, the prior axisymmetric assumptions should hold to at least some degree, and the measurements were made on this basis at the midpoint of each shaft.  At the locations of interest, the four samples had diameters as given in Table \ref{RomanToolsDiameter}.

\begin{figure}[!hbt]
    \centering
    \includegraphics[width=0.75\linewidth]{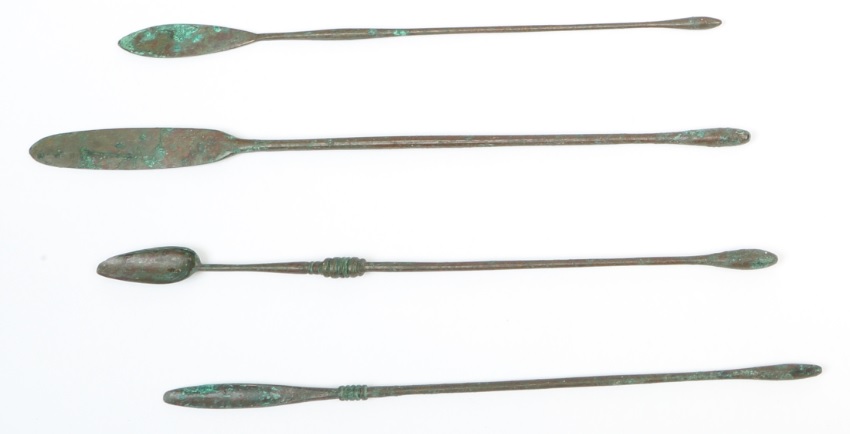}\\
    \caption{Roman bronze medical tools dating to 1-200CE Italy (from \cite{RomanToolsURL})}
    \label{romantools}
\end{figure}

\begin{table}
    \centering
    \begin{tabular}{cc}
        Sample & Diameter [mm]\\
        \hline\hline
        1 & 3.0\\
        2 & 2.2\\
        3 & 2.2\\
        4 & 1.6
    \end{tabular}
    \caption{Diameters of the cylindrical sections of the four Roman medical tools}
    \label{RomanToolsDiameter}
\end{table}

\begin{figure}[!hbt]
    \centering
    \includegraphics[width=0.5\linewidth]{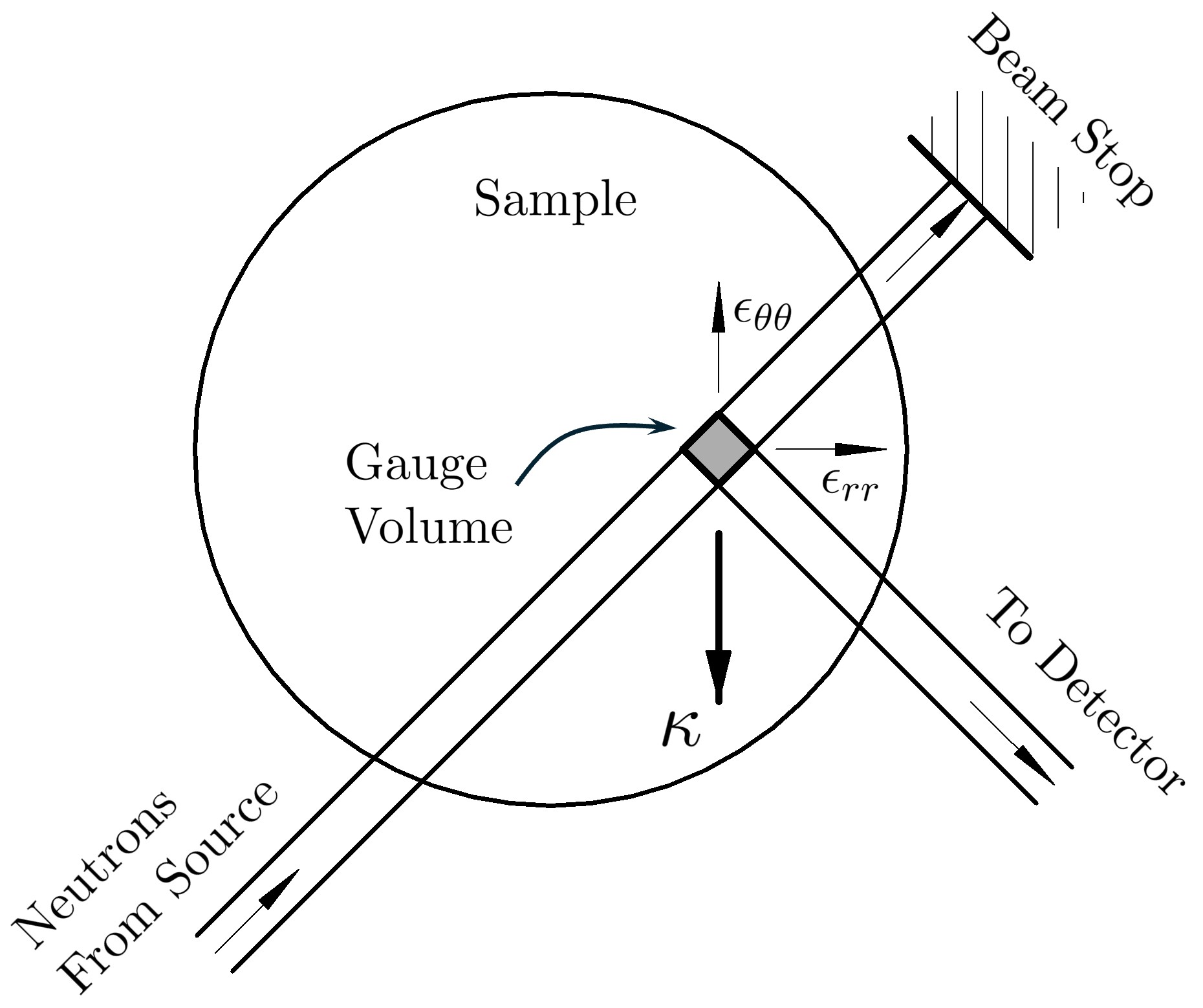}\\
    \caption{Geometry of strain measurements on the KOWARI instrument at ANSTO.  Through precise measurement of the diffraction angle to the detector, normal strain in the direction of $\kappa$ can be measured within the gauge volume.}
    \label{KOWARI}
\end{figure}

Details of the neutron-based strain measurement technique can be found elsewhere (e.g. \cite{hutchings2012measurement,withers2007mapping,hendriks2019robust}).  Briefly, the process relies upon measuring elastic strain as indicated by relative change in the lattice spacing of the crystal structure within the sample.  This involves accurately measuring certain angles of diffraction for neutrons of a given wavelength which are then related to lattice spacing through Bragg's Law.  With reference to Figure \ref{KOWARI}, each individual strain measurements is of the form
\[
\langle \epsilon_{ij}\kappa_i\kappa_j \rangle =\frac{d-d_0}{d_0}
\]
where $\langle \cdot \rangle $ denotes an average within a small, well-defined, gauge volume, $\kappa$ is a measurement direction relative to the sample, $d$ is a measured lattice spacing, and $d_0$ is a reference lattice spacing corresponding to unstressed material.  From a set of measurements over a sufficient number of directions, the full strain tensor can be determined at a series of points within the sample.  It is important to note that these strain measurements refer only to the elastic portion of strain and hence relate directly to stress through Hooke's law.

The strain measurements from the experiment consisted of radial, hoop and axial components at a series of points across each diameter.  These measurements were based on the relative shift of the (311) diffraction peak in copper for neutrons of wavelength 1.56\AA. Gauge volumes were $0.3 \times 0.3 \times 20$ mm, for radial and hoop components, and $0.3 \times 0.3 \times 0.5$ mm for axial.  Corresponding components of the stress tensor at each point were calculated using Hooke's law with elastic constants for tin-bronze estimated to be $E=130$ GPa and $\nu=0.34$.  

The resulting stress distributions are shown as the data points in Figure \ref{AxisymSoln} and \ref{Fit_d0}.  Broadly speaking, the observed distributions are consistent with either cold forging (hammering), or casting, as opposed to drawing.  As an exercise, we now analyse these data sets from the perspective of eigenstrain theory.

\begin{figure}[!htb]
    \centering
    \includegraphics[width=0.4\linewidth]{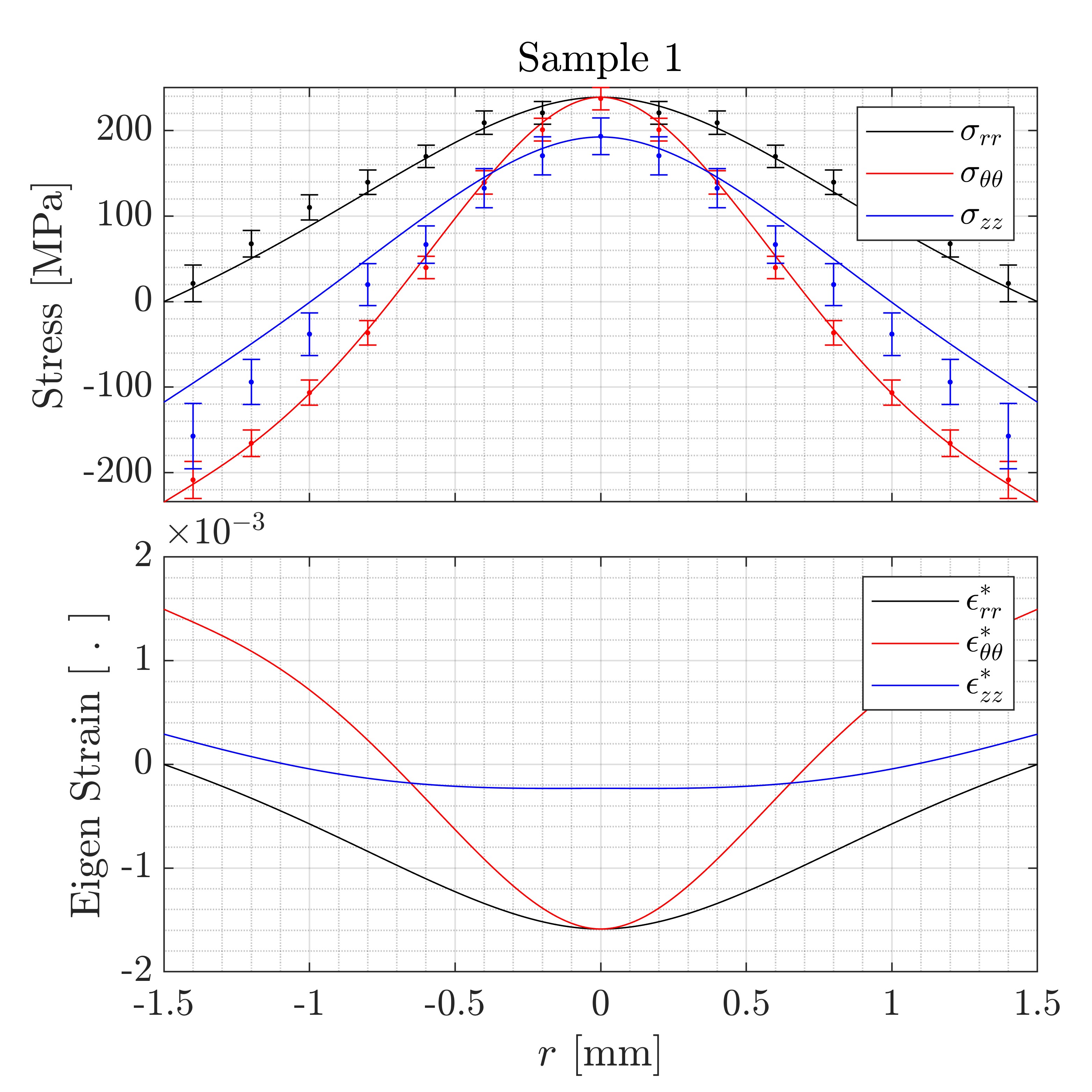}
    \includegraphics[width=0.4\linewidth]{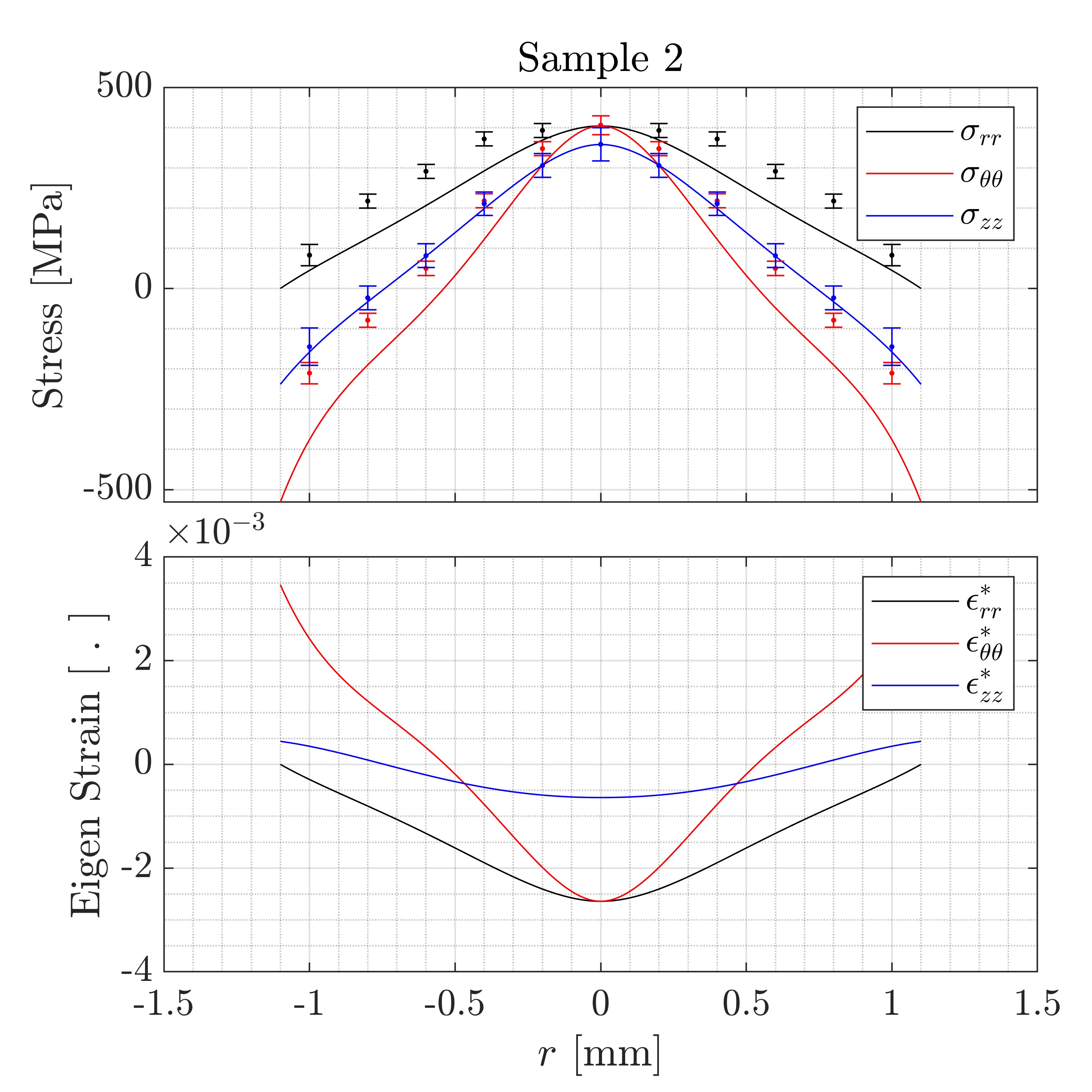}
    \includegraphics[width=0.4\linewidth]{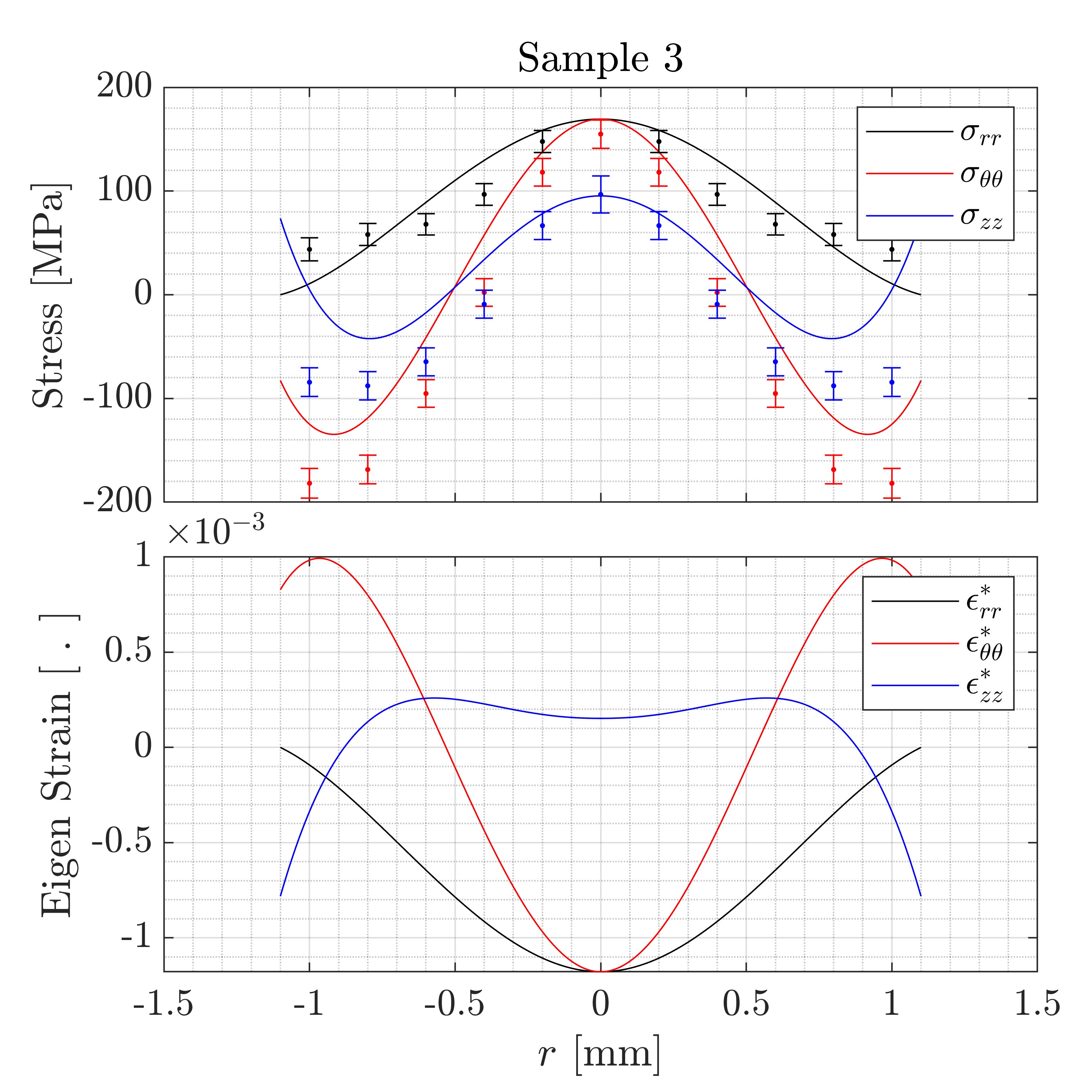}
    \includegraphics[width=0.4\linewidth]{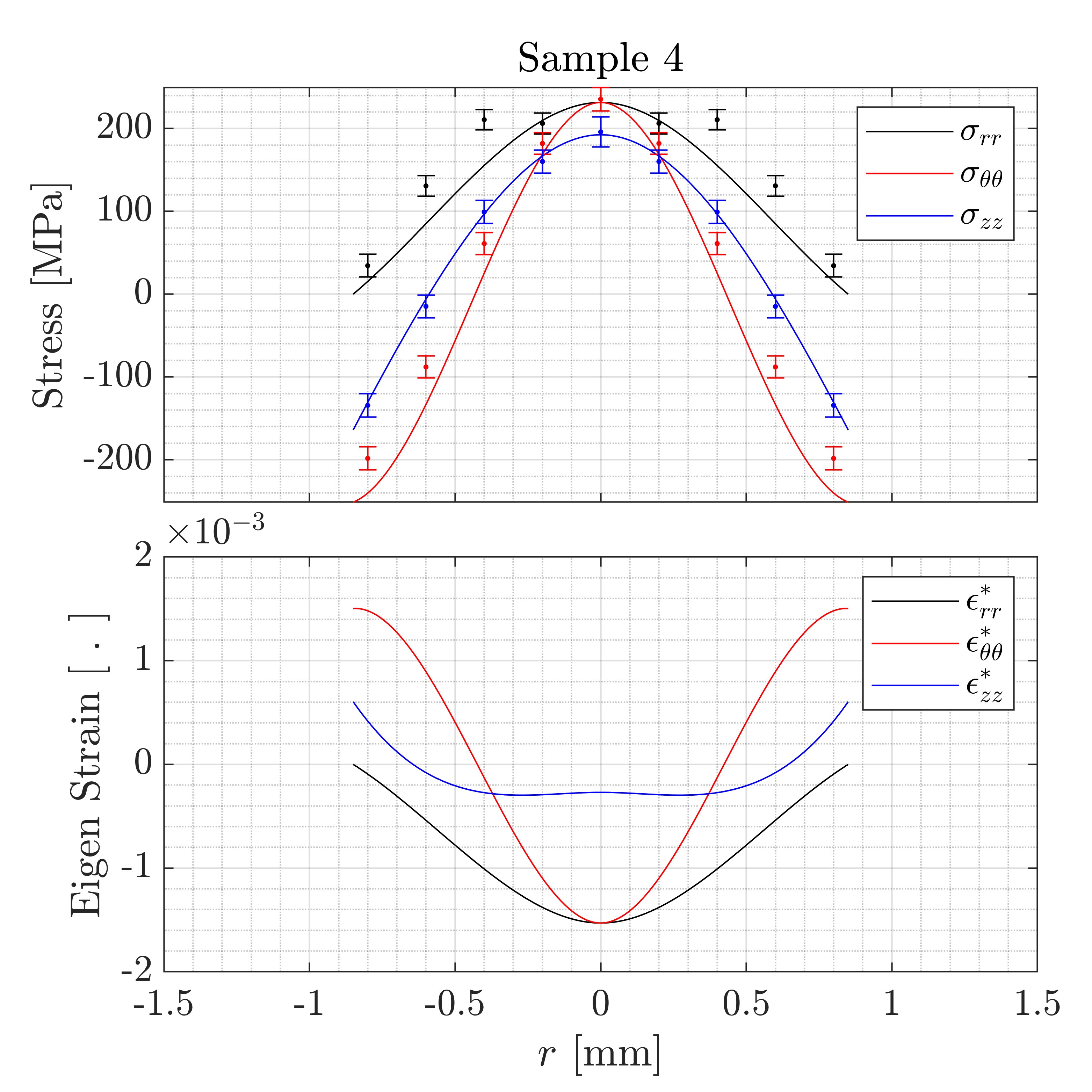}
    \caption{Eigenstrain solutions for the residual stress fields in the four Roman utensils compared against the measurements from KOWARI.}
    \label{AxisymSoln}
\end{figure}

\subsection{Polynomial eigenstrain solutions}

4$^{th}$ order polynomial coefficients for $\epsilon_{rr}^*$, $\epsilon_{\theta\theta}^*$, and $\epsilon_{zz}^*$ were determined through least squares minimisation of the error between the measured residual stress profiles and predictions from the axisymmetric model detailed above.  This optimisation was carried out under the following conditions;

\begin{enumerate}
    \item{Coefficients of linear terms were set to zero on the basis that gradients should vanish at the origin of the axisymmetric system.}
    \item{Null components were excluded by applying the constraints outlined in \eqref{PolyConstraints} to the coefficients.}
\end{enumerate}

\begin{figure}[!htb]
    	\centering
        \includegraphics[width=0.4\linewidth]{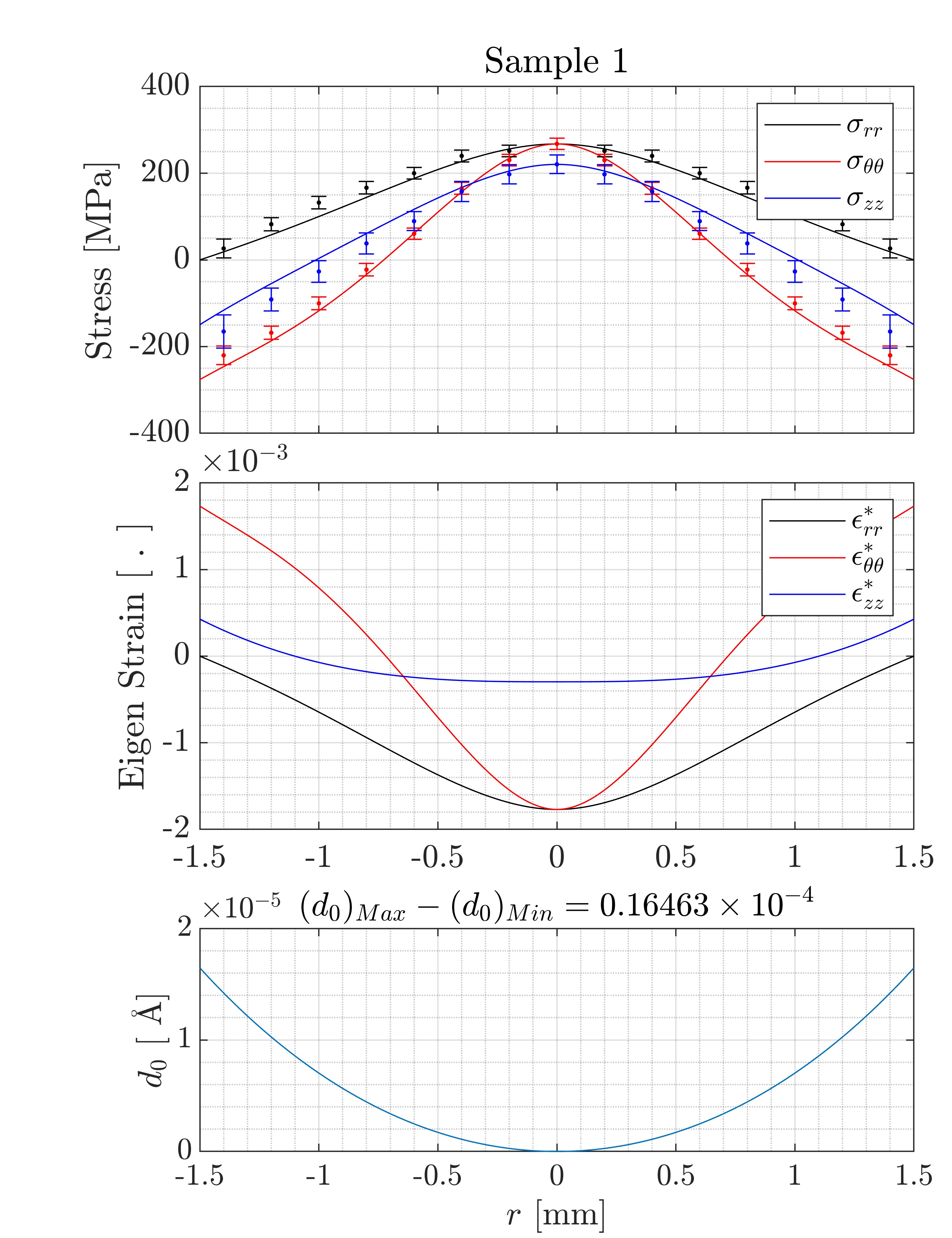}
        \includegraphics[width=0.4\linewidth]{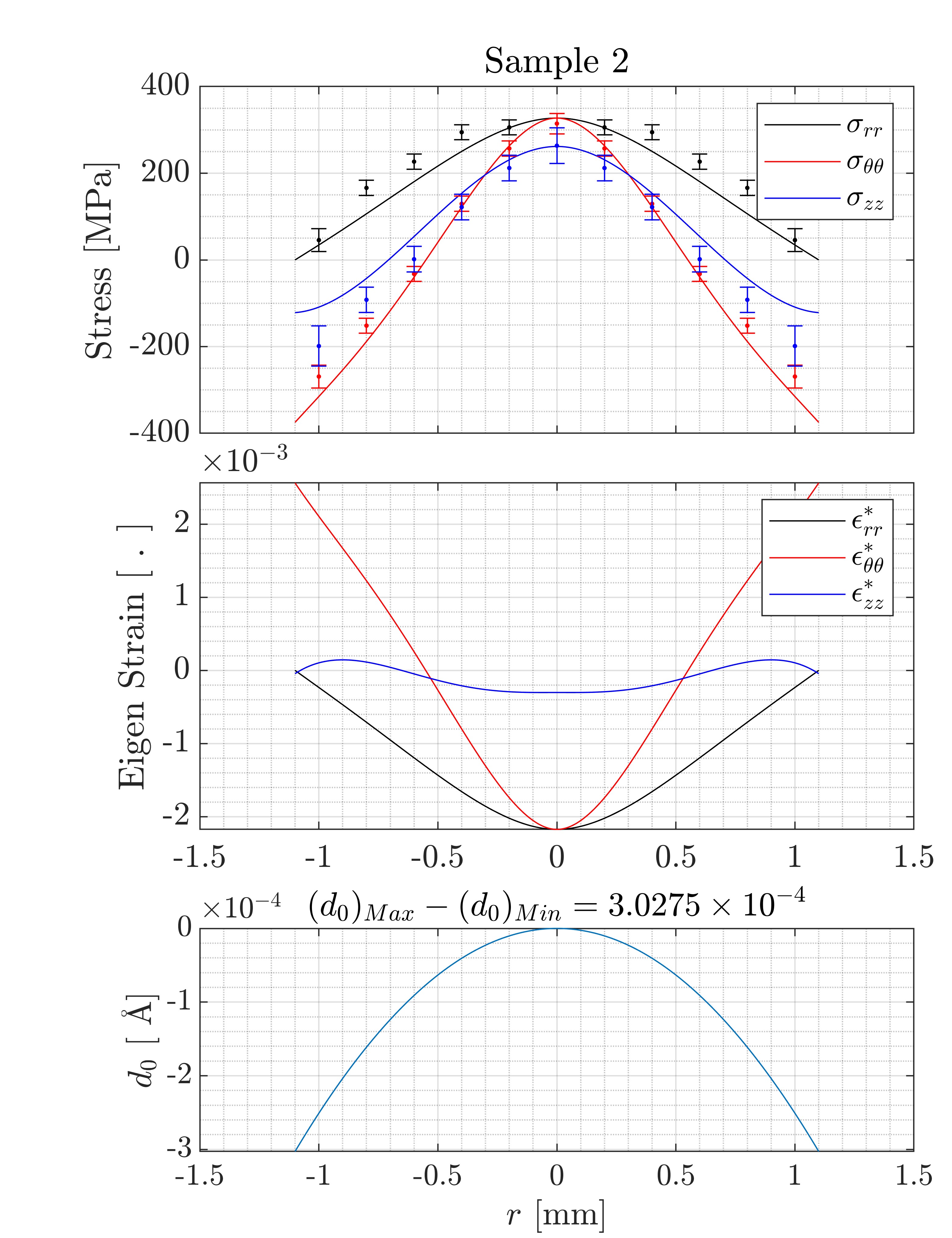}
        \includegraphics[width=0.4\linewidth]{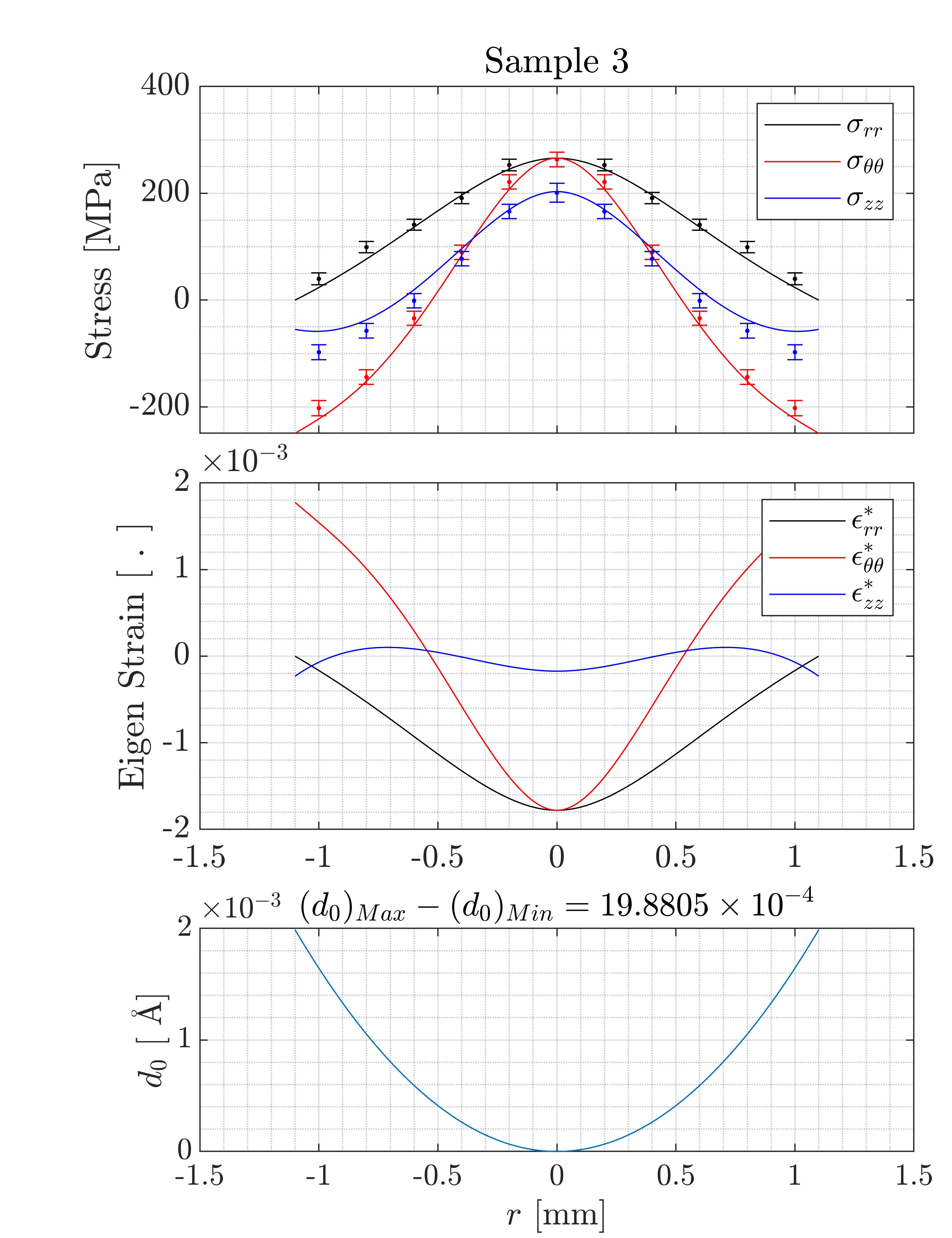}
        \includegraphics[width=0.4\linewidth]{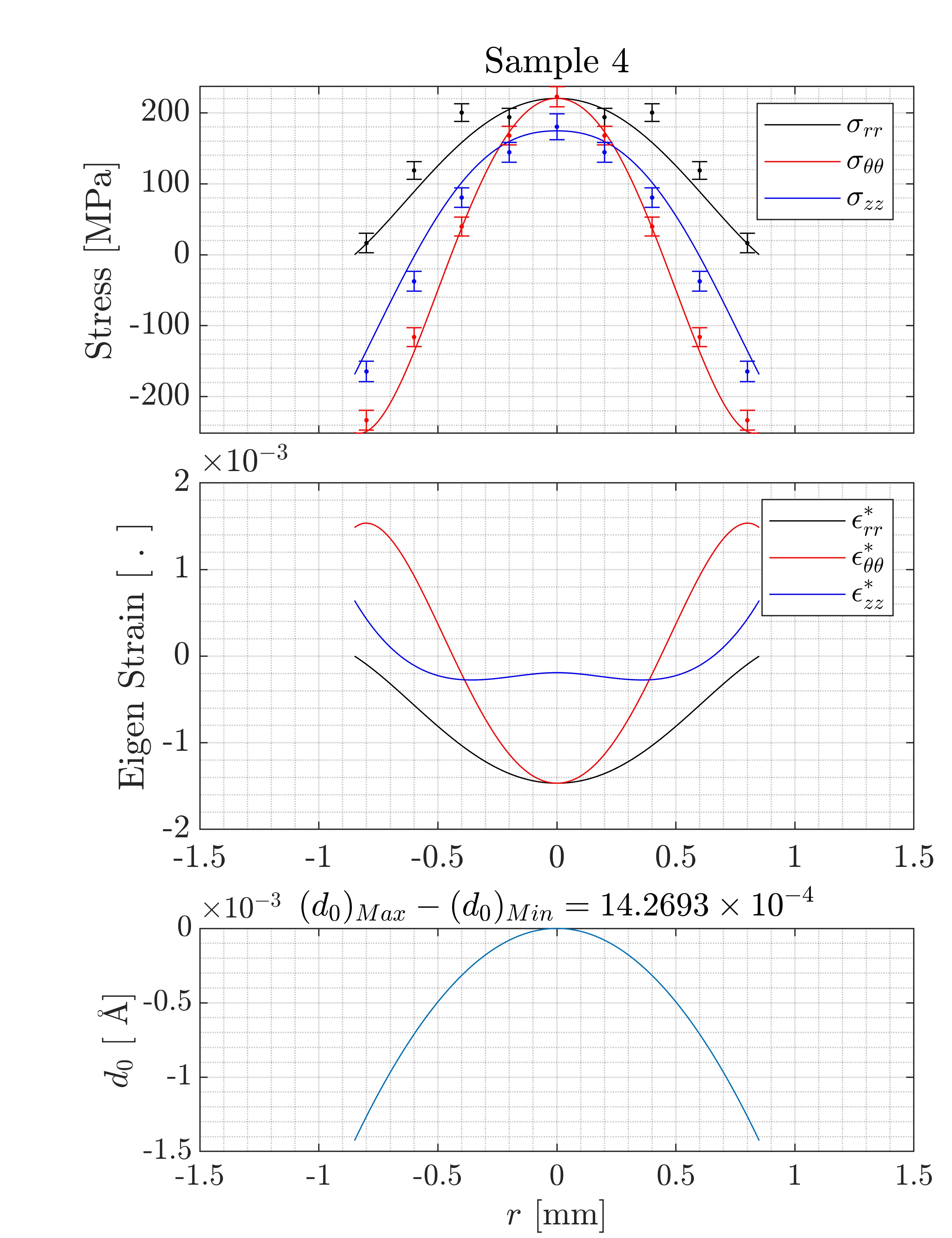}
    	\caption{Eigenstrain solutions for the Roman utensils computed alongside a variable isotropic $d_0$ that was assumed to be a function of radius.}
        \label{Fit_d0}
\end{figure}

The optimisation was implemented using the Levenberg-Marquardt algorithm within the `\texttt{lsqnonlin}' function in the MATLAB proprietary Optimisation Toolbox.  Figure \ref{AxisymSoln} shows the corresponding eigenstrain solution for each sample plotted alongside the corresponding experimental data. Observations from these figures include;
\begin{enumerate}
    \item{Sample 1 shows excellent agreement between the measured stress and the eigenstrain solution.  The same could be said about sample 4 to some degree.  Samples 2 and 3 show some similarity but with significant discrepancy.}
    \item{The form of the eigenstrain distributions is relatively consistent between samples.  Largely speaking, all eigenstrain components were observed to be monotonically increasing with $r$ with positive curvature over most of the range.  The only exception to this was $^s\epsilon_{\theta\theta}^*$ near the outer radius in samples 3 and 4 and $^s\epsilon_{zz}^*$ in sample 3.}
    \item{For all samples, variation of $\epsilon^*_{\theta\theta}$ over the diameter was larger than that of $\epsilon^*_{rr}$ and $\epsilon^*_{zz}$ respectively.}
\end{enumerate}

\subsection{Simultaneous estimation of $d_0$}

The discrepancy between the model and the data in samples 2, 3 and 4 was thought to be due to variation in $d_0$ within the sample.  Values for $d_0$ were not measured directly from these samples; typical destructive methods cannot be used here.  Rather, a single value of $d_0$ was chosen for each sample through a process of ensuring stress-balance.  In reality, $d_0$ is dependent upon various factors such as material composition,  defect density within the crystal structure, and several other factors that can vary throughout the sample.

In an attempt to account for this issue, the above analysis was modified to allow for a variable isotropic $d_0$ as a function of radius that was parameterised by an additional polynomial. The coefficients of this polynomial were chosen through the same least squares optimisation process alongside the eigenstrain coefficients.  Results of this process are shown in Figure \ref{Fit_d0}.

Compared to Figure \ref{AxisymSoln}, we see that the overall fit between the eigenstrain solutions and the measured data is dramatically improved; particularly for samples 2 and 3 but also sample 4.  For sample 1, variation in the reconstructed distribution of $d_0$ was less than the typical level of uncertainty which is of the order of $0.6 \times 10^{-4}$\AA.  This is not the case for the other samples where the reconstructed $d_0$ showed significant variation over the diameter; up to 30 times the level of uncertainty, but still within the bounds of reasonable physical variation.

Unfortunately no data exists to validate these distributions of $d_0$; however the example serves to demonstrate that the constraint of equilibrium (as implemented through the eigenstrain framework) is powerful and may allow for the extraction of otherwise hidden information within experimental data.

\section{Bounded Helmholtz decomposition of eigenstrain}

In a more general setting, we can use a similar process to define an orthogonal range-null decomposition of eigenstrain fields on $\Omega$ with respect to the usual inner product.  As before, null eigenstrains have the form $\epsilon^*=\nabla_sU^*$ for some vector potential $U^*\in H^1(\mathcal{S}^1;\Omega)$. The orthogonal complement is then formed from all $\epsilon^{*\perp}$ such that
\[
\langle \epsilon^{*\perp}, \nabla_sU^* \rangle = \int_\Omega \epsilon^{*\perp} : \nabla_sU^* \hspace{1ex} \mathrm{d}V = 0.
\]
Integrating by parts, we can write this as
\[
\int_{\partial\Omega} (\epsilon^{*\perp} \cdot n) \cdot U^* \hspace{1ex} \mathrm{d}A \quad - \quad \int_{\Omega} \text{Div }(\epsilon^{*\perp}) \cdot U^* \hspace{1ex} \mathrm{d}V \quad = \quad  0.
\]
Given the arbitrary nature of $U^*$, this implies
\begin{equation}
\label{DecompPDE}
\epsilon^{*\perp} \cdot n \Big|_{\partial\Omega}=0 \text{, and } \text{Div }\epsilon^{*\perp} = 0,
\end{equation}
and hence we can write any eigenstrain as an orthogonal decomposition
\begin{equation}
\label{HelmholtzDecomp}
    \epsilon^*=\nabla_sU^*+\epsilon^{*\perp}
\end{equation}
where $\nabla_s U^*$ is a curl-free potential (in the null-space of the eigenstrain problem), and $\epsilon^{*\perp}$ is divergence-free `solenoid' with zero boundary flux on $\partial\Omega$.  This should be recognised as the usual Helmholtz decomposition of tensor fields on a bounded domain \cite{sharafutdinov2012integral}.

As with the axisymmetric example, we can exclude the null-space to dramatically reduce the size of the inverse problem, however implementation of this requires a practical means for paramaterising an arbitrary solenoidal eigenstrain.  In this endeavor, we take inspiration from the concept of a `stress function' \cite{langhaar1954three,pommaret2015airy} which allows for the construction of an arbitrary solenoidal field from a Beltrami potential $\Lambda \in H^2(\mathcal{S}^2;\mathbb{R}^3)$ as
\[
\epsilon^{*\perp}=\mathcal{R}\Lambda=\nabla \times (\nabla \times \Lambda)^T.
\]
This eigenstrain is guaranteed to be divergence-free regardless of the choice of $\Lambda$.  

A simplification based on a Maxwell potential \cite{pommaret2015airy} restricts $\Lambda$ to a diagonal form in terms of three arbitrary scalar fields
\[
\Lambda=\begin{bmatrix}
    \Lambda_x & 0 & 0\\
    0 & \Lambda_y & 0\\
    0 & 0 & \Lambda_z
\end{bmatrix},
\]
each within $H^2(\Omega)$.  From this we can define the divergence-free eigenstrain
\begin{equation}
\label{MaxwellEigenstrain}
\epsilon^{*\perp}=\begin{bmatrix}
            \frac{\partial^2\Lambda_y}{\partial z^2}+\frac{\partial^2\Lambda_z}{\partial y^2}& -\frac{\partial^2\Lambda_z}{\partial x \partial y} & -\frac{\partial^2\Lambda_y}{\partial x\partial z} \\
            -\frac{\partial^2\Lambda_z}{\partial x \partial y} & \frac{\partial^2\Lambda_z}{\partial x^2}+\frac{\partial^2\Lambda_x}{\partial z^2} & -\frac{\partial^2\Lambda_x}{\partial y\partial z} \\
            -\frac{\partial^2\Lambda_y}{\partial x \partial z} & -\frac{\partial^2\Lambda_x}{\partial y\partial z} & \frac{\partial^2\Lambda_x}{\partial y^2}+\frac{\partial^2\Lambda_y}{\partial x^2}
        \end{bmatrix}.
\end{equation}

The inverse problem then becomes the parameterisation of the potentials $\Lambda_x$, $\Lambda_y$ and $\Lambda_z$ in a way that satisfies the zero boundary flux condition, and then the selection of the parameters to match an observed stress field.  In this regard, it should be pointed out that the operator $\mathcal{R}$ has its own null space (i.e. potential fields), however the corresponding solenoidal eigenstrain for a given stress field will be unique.  It should also be noted that the span of $\mathcal{R}$ when restricted to Maxwell potentials covers all divergence-free tensor fields \cite{langhaar1954three}; any solenoidal eigenstrain can be written in the form of \eqref{MaxwellEigenstrain} for some non-unique $\Lambda_x$, $\Lambda_y$ and $\Lambda_z$.

In the following section we provide an example of this process.

\section{Example 2: Additive manufacturing}
\label{AMSection}

\subsection{Sample description and neutron-based stress measurement}

Additive manufacturing refers to a variety of processes that involve 3D `printing' of components through layer-by-layer deposition of material.  Typically, the deposition process for metals involves high temperatures, phase changes and rapid cooling rates that combine to embed eigenstrain into the final component. Variations in this thermo-mechanical process over the printed volume can result in substantial residual stress. Significant interest exists in modeling these processes and understanding the underlying eigenstrain fields that are produced.

In this vein, Figure \ref{SLMCubes} shows a matching pair of additively manufactured $17\times 17 \times 17$ mm cubes printed from Inconel-718 by Selective Laser Melting (SLM).  These cubes were the focus of a study on residual stress created by this manufacturing process which involved direct measurements of the internal strain field using the KOWARI instrument.  Full details of the printing process/parameters and measurement procedure can be found in Wensrich \emph{et al.} \cite{wensrich2024residual}, however the salient features for us here include;

\begin{enumerate}
    \item{The direction of the raster pattern of the laser during printing was rotated by 67$^\circ$ between each layer.  This means that, in terms of the resulting stress distribution, the $x$ and $y$-directions are arbitrary and $z$ is distinct.}
    \item{All components of the stress tensor were measured over a regular $ 8\times 8$ grid on a central cross section of the cube as shown in Figure \ref{SLMCubes}.  This involved measuring normal strain in 9 separate directions at each grid point from which the 6 components of stress could be calculated through a least squares process.  Each individual measurement was based on the relative shift of the (311) diffraction peak in nickel with neutrons of wavelength of $\lambda=1.52$\AA and a $1\times 1 \times 1$ mm gauge volume.  Elastic properties for Inconel-718 were assumed to be $E=208$ GPa, and $\nu=0.28$ for this process.}
    \item{The measured stress distribution was based on a $d_0$ measured from a \emph{comb} sample from an equivalent central section of an identical cube.  This comb sample was cut from the second cube using Electrical Discharge Machining (EDM) and then subject to a series of vertical EDM cuts to form a row of stress-free fingers from which an array $d_0$ of measurements were made from several directions. No spatial or anisotropic variation in $d_0$ was observed.}
\end{enumerate}

\begin{figure}
    \centering
    \includegraphics[width=0.32\linewidth]{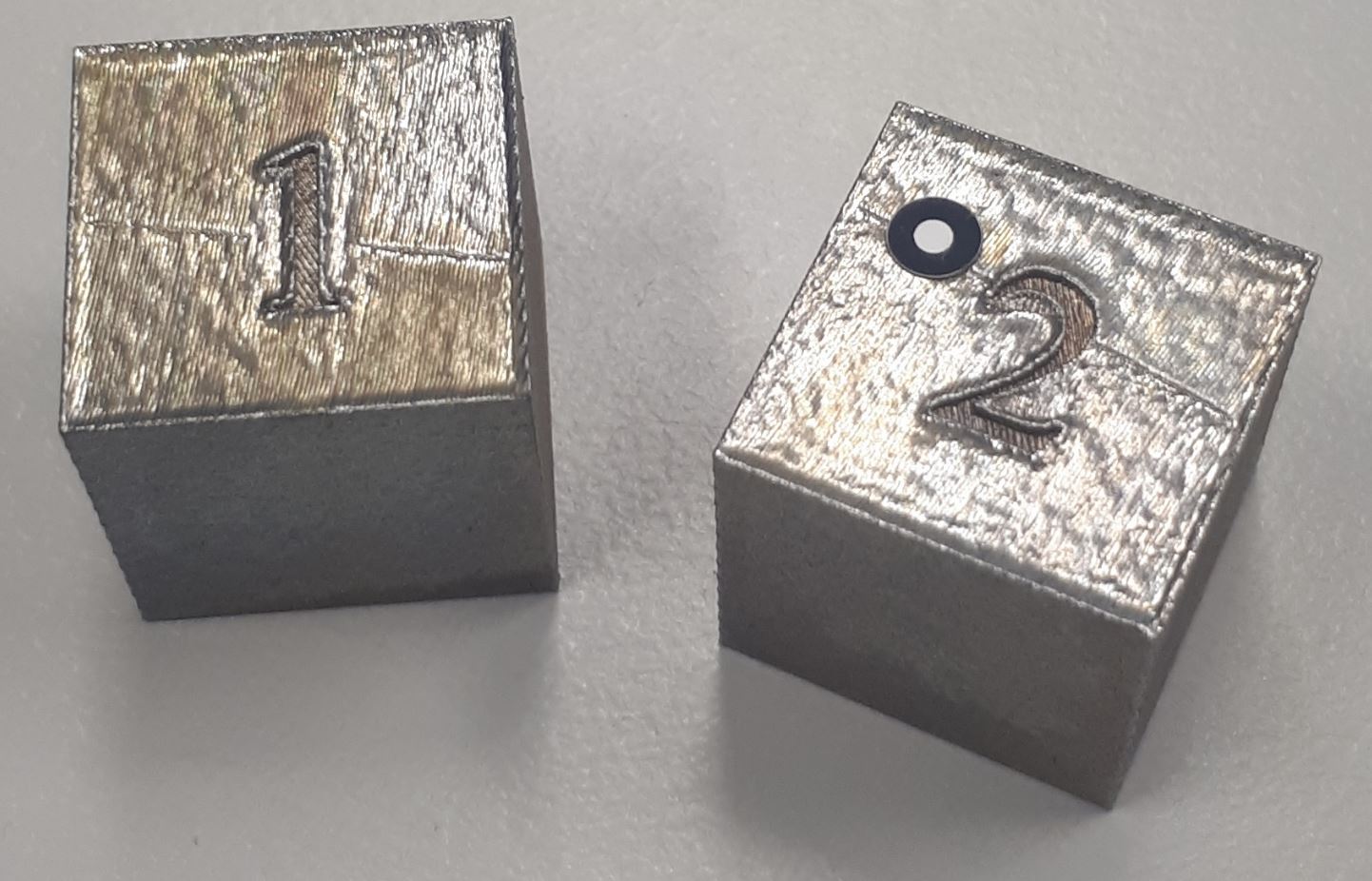}
    \includegraphics[width=0.38\linewidth]{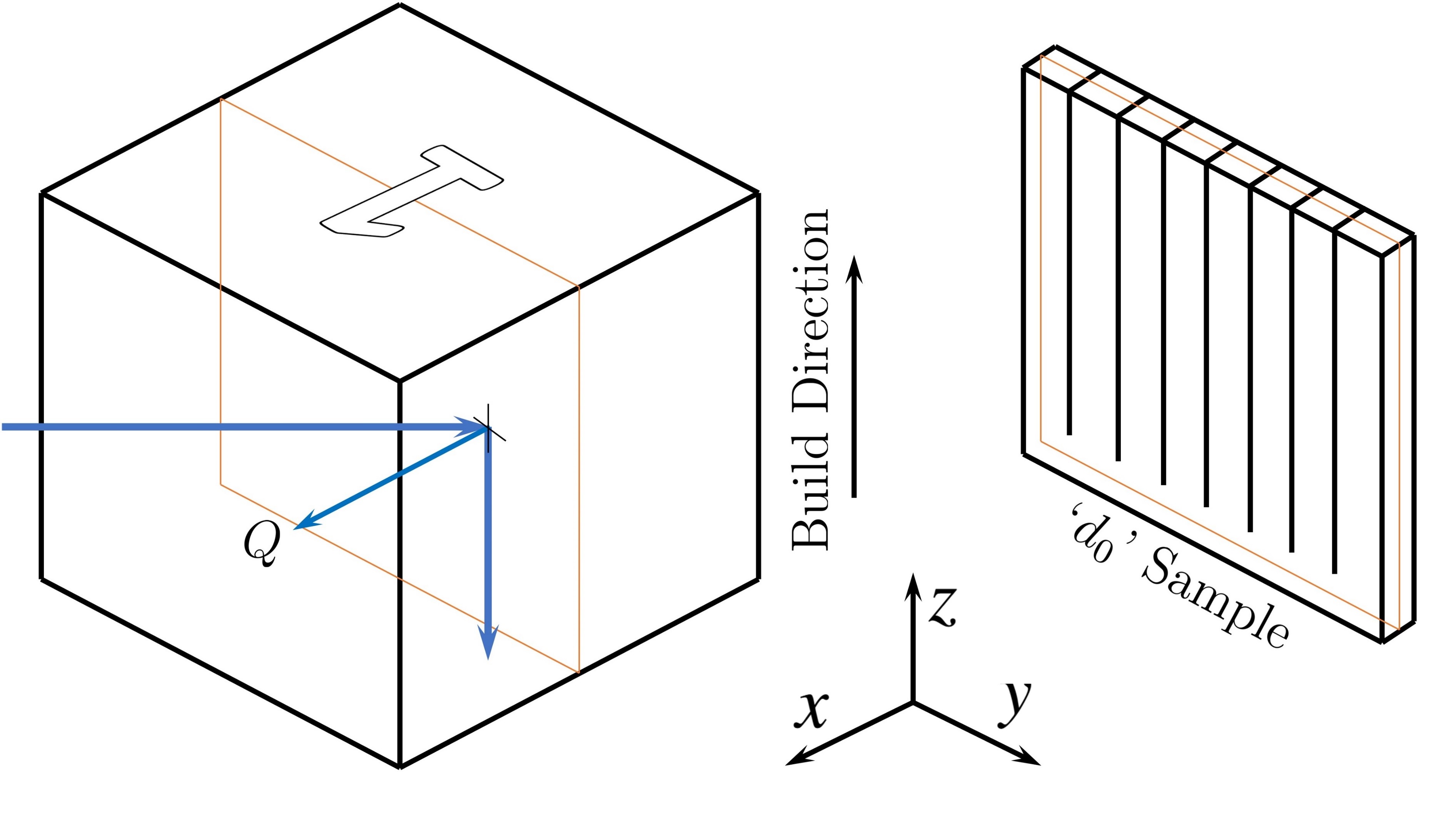}
    \put(-270,82){(a)}
    \put(-145,82){(b)}
    \hspace{0.5ex}
    \includegraphics[width=0.22\linewidth]{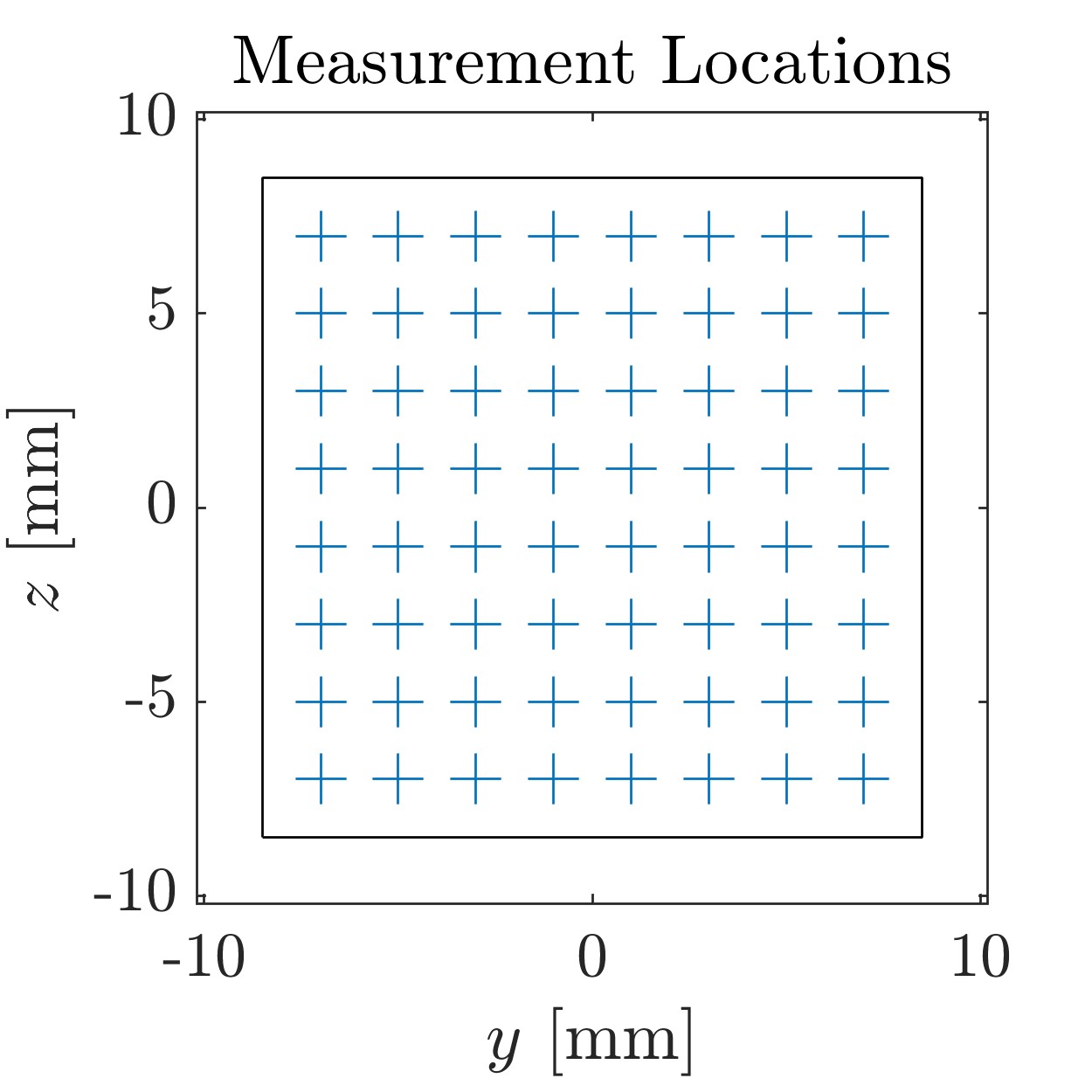}
    \put(-85,82){(c)}
    \caption{(a) A pair of identical $17\times 17 \times 17$ mm Inconel cubes produced using Selective Laser Melting \cite{wensrich2024residual}  (b) Cube 1 was the subject of residual stress measurements over a central cross section using the KOWARI strain diffractometer at ANSTO.  The schematic shows the geometry for the measurement of the $\epsilon_{xx}$ component.  Cube 2 was dissected using Electro-Discharge Machining (EDM) to provide a stress-free `$d_0$' comb. (c) Locations of the experimental measurements over the cross section.}
    \label{SLMCubes}
\end{figure}

The measured residual stress distribution is shown in Figure \ref{SLMCubeStress}a.

\subsection{Direct modeling of stress}

We begin by modeling the stress distribution in the cube directly using a Maxwell potential.  Similar to before, this implies that any divergence-free stress field can be written in the form
\begin{equation}
\label{MaxwellStress}
\sigma=\mathcal{R}\Lambda=\begin{bmatrix}
            \frac{\partial^2\Lambda_y}{\partial z^2}+\frac{\partial^2\Lambda_z}{\partial y^2}& -\frac{\partial^2\Lambda_z}{\partial x \partial y} & -\frac{\partial^2\Lambda_y}{\partial x\partial z} \\
            -\frac{\partial^2\Lambda_z}{\partial x \partial y} & \frac{\partial^2\Lambda_z}{\partial x^2}+\frac{\partial^2\Lambda_x}{\partial z^2} & -\frac{\partial^2\Lambda_x}{\partial y\partial z} \\
            -\frac{\partial^2\Lambda_y}{\partial x \partial z} & -\frac{\partial^2\Lambda_x}{\partial y\partial z} & \frac{\partial^2\Lambda_x}{\partial y^2}+\frac{\partial^2\Lambda_y}{\partial x^2}
        \end{bmatrix},
\end{equation}
where the three scalar potentials $\Lambda_x$, $\Lambda_y$ and $\Lambda_z$ are now considered to be `stress functions'.  

The cubic geometry of the sample and details of the printing process imply a number of symmetries which can be exploited as constraints on these stress functions.  For example, the arbitrary nature of coordinates in the build plane implies that only two distinct scalar potentials are required.  We denote these as 
\[
\Lambda_{||}=\Lambda_x=\Lambda_y \quad \text{and} \quad \Lambda_\perp = \Lambda_z,
\]
with both $\Lambda_{||}$ and $\Lambda_\perp$ symmetric under the exchange of $x$ and $y$.   The mirror symmetries about the $x=0$, $y=0$ and $z=0$ planes also imply that both of these functions are even in $x$, $y$ and $z$.  Furthermore, the zero boundary traction/flux constraint \eqref{NoTraction} implies a number of conditions of the form
\begin{align*}
\frac{\partial^2\Lambda_{||}}{\partial z^2}+\frac{\partial^2\Lambda_\perp}{\partial y^2}\Big|_{x=L} &= 0\\
\frac{\partial^2\Lambda_{||}}{\partial x^2}+\frac{\partial^2\Lambda_{||}}{\partial y^2}\Big|_{z=L} &= 0\\
\frac{\partial^2\Lambda_\perp}{\partial x\partial y} \Big|_{x=L} &= 0\\
\frac{\partial^2\Lambda_{||}}{\partial x\partial z} \Big|_{x=L} &= 0\\
\frac{\partial^2\Lambda_{||}}{\partial x\partial z} \Big|_{z=L} &= 0,
\end{align*}
where $L$ is the size of the cube (i.e. $x\in[-L,L]$, $y\in[-L,L]$, and $z\in[-L,L]$). Note that a number of similar statements can be obtained through exchange of $x$ and $y$.  All of these conditions together imply that $\Lambda$ is factored by 
\[
\Phi=\Big( \big(\tfrac{x}{L}\big)^2-1\Big)^2\Big( \big(\tfrac{y}{L}\big)^2-1\Big)^2\Big( \big(\tfrac{z}{L}\big)^2-1\Big)^2.
\]

Hence, relying on the fundamental theorem of symmetric polynomials \cite{blum2017fundamental}, we can write general polynomials for $\Lambda_{||}$ and $\Lambda_{\perp}$ as
\begin{align*}
\Lambda_{||}&=\Phi\Big(a_0+a_1\big(\tfrac{z}{L}\big)^2+a_2\big(\tfrac{z}{L}\big)^4+\dots\Big)\\
\Lambda_{\perp}&=\Phi\Big(b_0+b_1\big(\tfrac{z}{L}\big)^2+b_2\big(\tfrac{z}{L}\big)^4+\dots\Big),
\end{align*}
where the $a$ and $b$ coefficients are symmetric polynomial functions of the form
\begin{center}
\vspace{-4ex}
\begin{minipage}{0.4\textwidth}
	\begin{align*}
	a_i(x,y)=&a_{i0}+\\
			&a_{i1}e_1+\\
			&a_{i2}e_1^2+a_{i3}e_2+\\
			&a_{i4}e_1^3+a_{i5}e_1e_2+\\
            &a_{i6}e_1^4+a_{i7}e_1^2e_2+a_{i8}e_2^2+\dots
	\end{align*}
\end{minipage}
\begin{minipage}{0.4\textwidth}
	\begin{align*}
	b_i(x,y)=&b_{i0}+\\
			&b_{i1}e_1+\\
			&b_{i2}e_1^2+b_{i3}e_2+\\
			&b_{i4}e_1^3+b_{i5}e_1e_2+\\
			&b_{i6}e_1^4+b_{i7}e_1^2e_2+b_{i8}e_2^2+\dots
	\end{align*}
\end{minipage}
\end{center}
with $e_1$ and $e_2$ the elementary symmetric polynomials
\[
e_1=\tfrac{1}{L^2}(x^2+y^2) \quad \text{and} \quad e_2=\tfrac{1}{L^4}x^2y^2.
\]

Through this parameterisation, we can form a least-squares best-fit to the observed stress through the following process;

Truncating the polynomials for $\Lambda_{||}$ and $\Lambda_\perp$ at a given order provides a set of basis vectors to describe stress.  For example, noting that $\Phi$ is of order 12,  order 24 polynomials (i.e. maximum order of 8 for $x$, $y$ and $z$ individually) for $\Lambda_{||}$ and $\Lambda_\perp$ can be written in terms of 24 basis functions 
 with 24 coefficients; $3 \times 4$ $a$'s and $3 \times 4$ $b$'s.  Each of these basis vectors for $\Lambda$ corresponds to a stress field defined over the whole cube via \eqref{MaxwellStress}, and in particular, defines the 6 unique components stress at the $N=8\times 8$ measurement locations from the experiment.  From this we form the following linear system for polynomials of order 24;
\[
A_{[6N \times 24]}
\begin{bmatrix}
    a\\
    b
\end{bmatrix}_{[24 \times 1]}=\sigma_{[6N \times 1]}
\]
where the right hand side is a vector containing all of the experimental measurements, and the columns of $A$ are formed by corresponding values from each basis vector at the same locations.  The least-squares solution for the coefficients (i.e. minimum $L^2$ norm) can then be computed via the Moore-Penrose pseudo-inverse of $A$.  The results of this process in terms of components of the fitted stress distribution over the central cross section are shown in Figure \ref{SLMCubeStress}b.  Also shown in Figure \ref{SLMCubeStress}c and d are profiles of the normal components along the $y$ and $z$ axes showing a direct comparison of measure stresses with the fitted Maxwell distribution.  Excellent agreement is observed within the confines of the measured data, however outside of this there is some minor evidence of Runge's phenomenon near the boundary.

\begin{figure}[!htb]
    	\centering
        \includegraphics[width=0.45\linewidth]{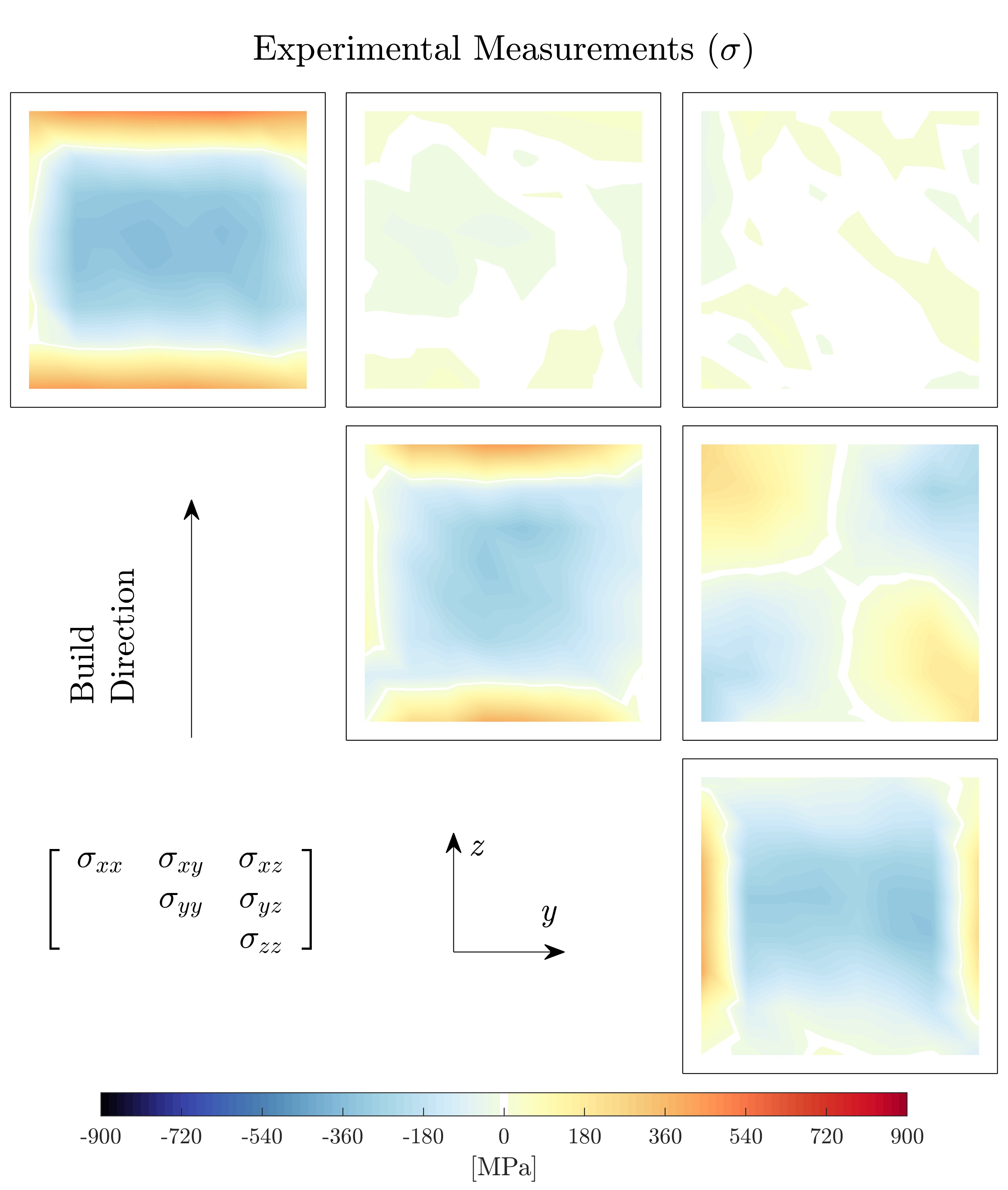}
        \put(-170,195){(a)}
        \includegraphics[width=0.45\linewidth]{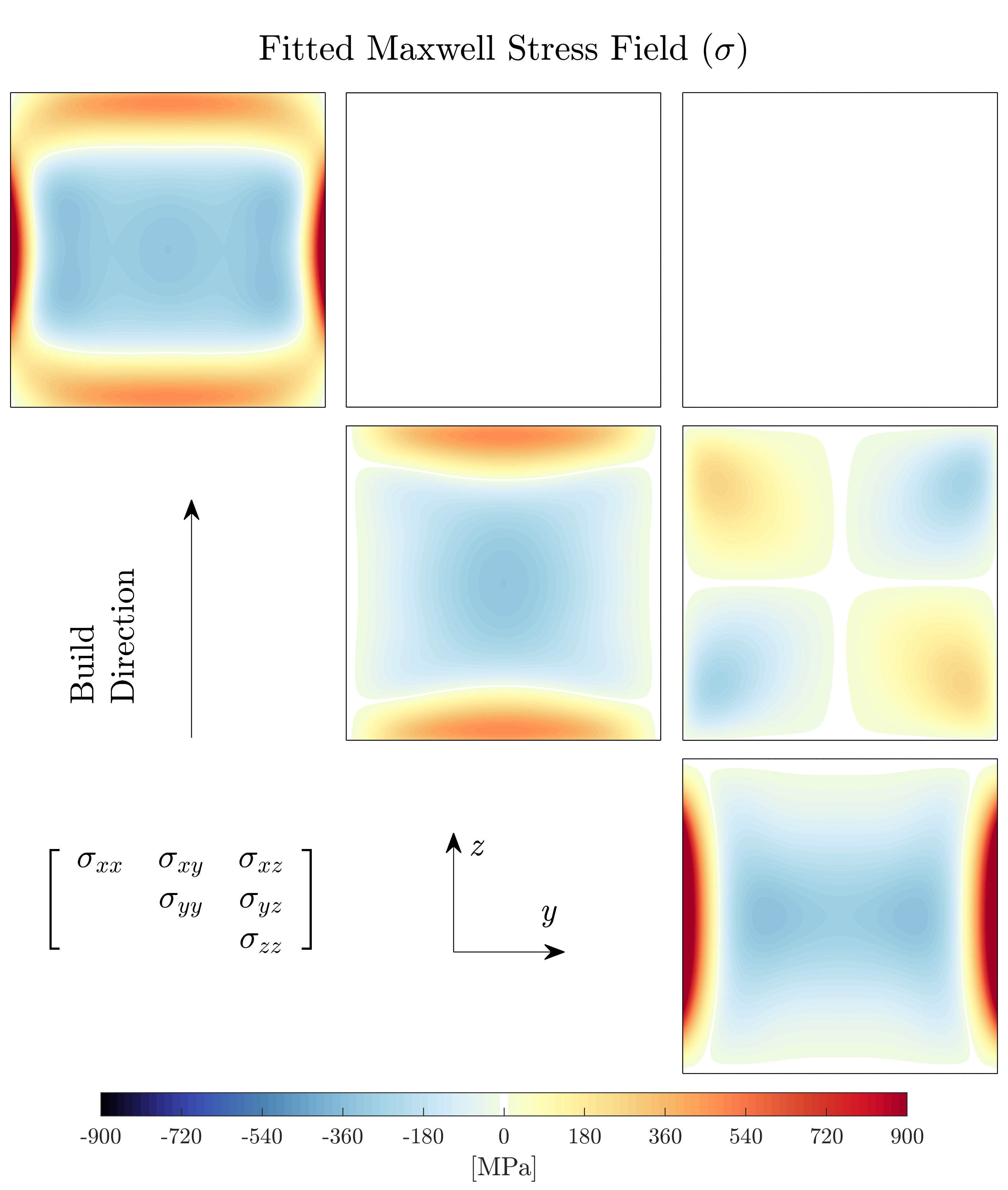}
        \put(-170,195){(b)}\\
        \includegraphics[width=0.45\linewidth]{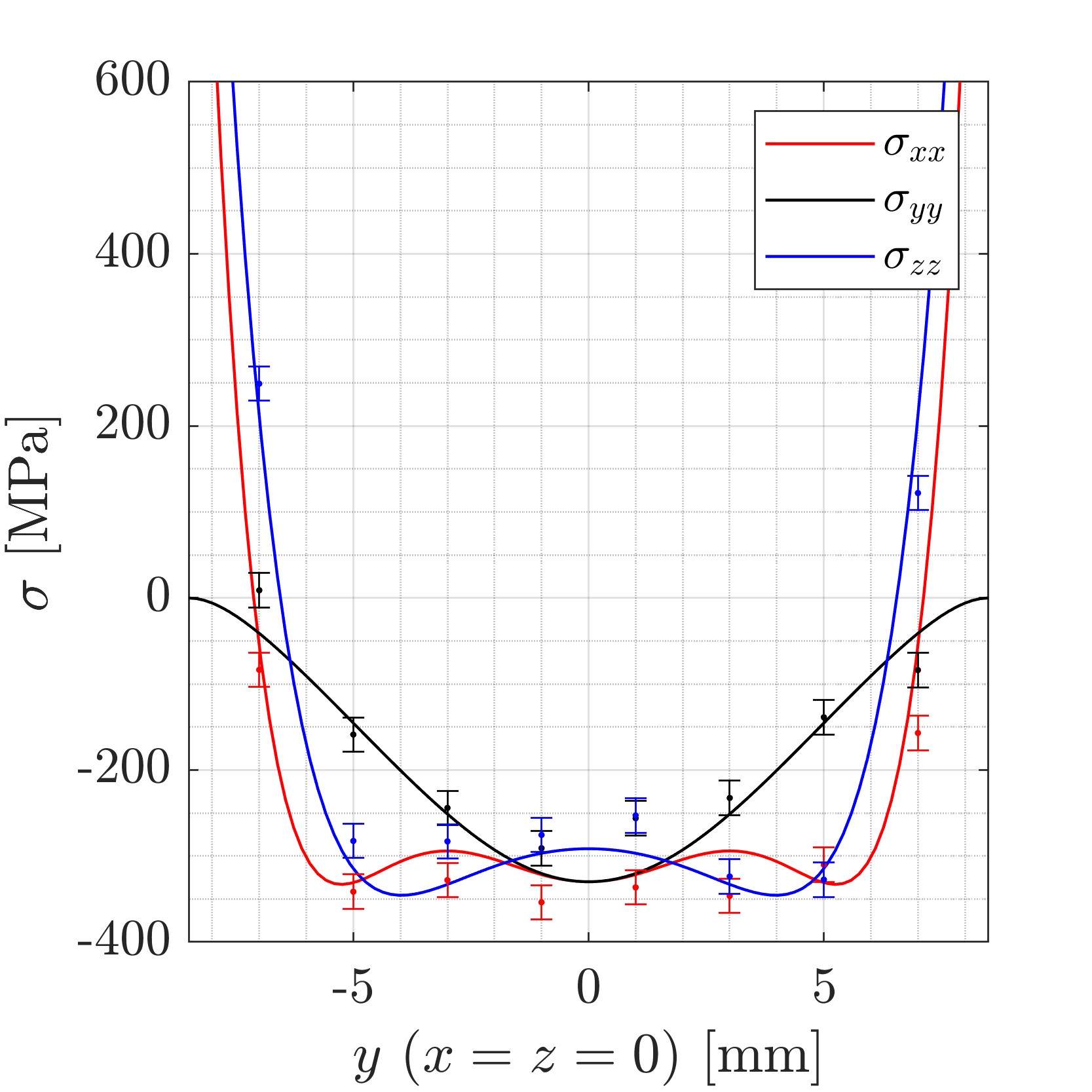}
        \put(-172,155){(c)}
        \includegraphics[width=0.45\linewidth]{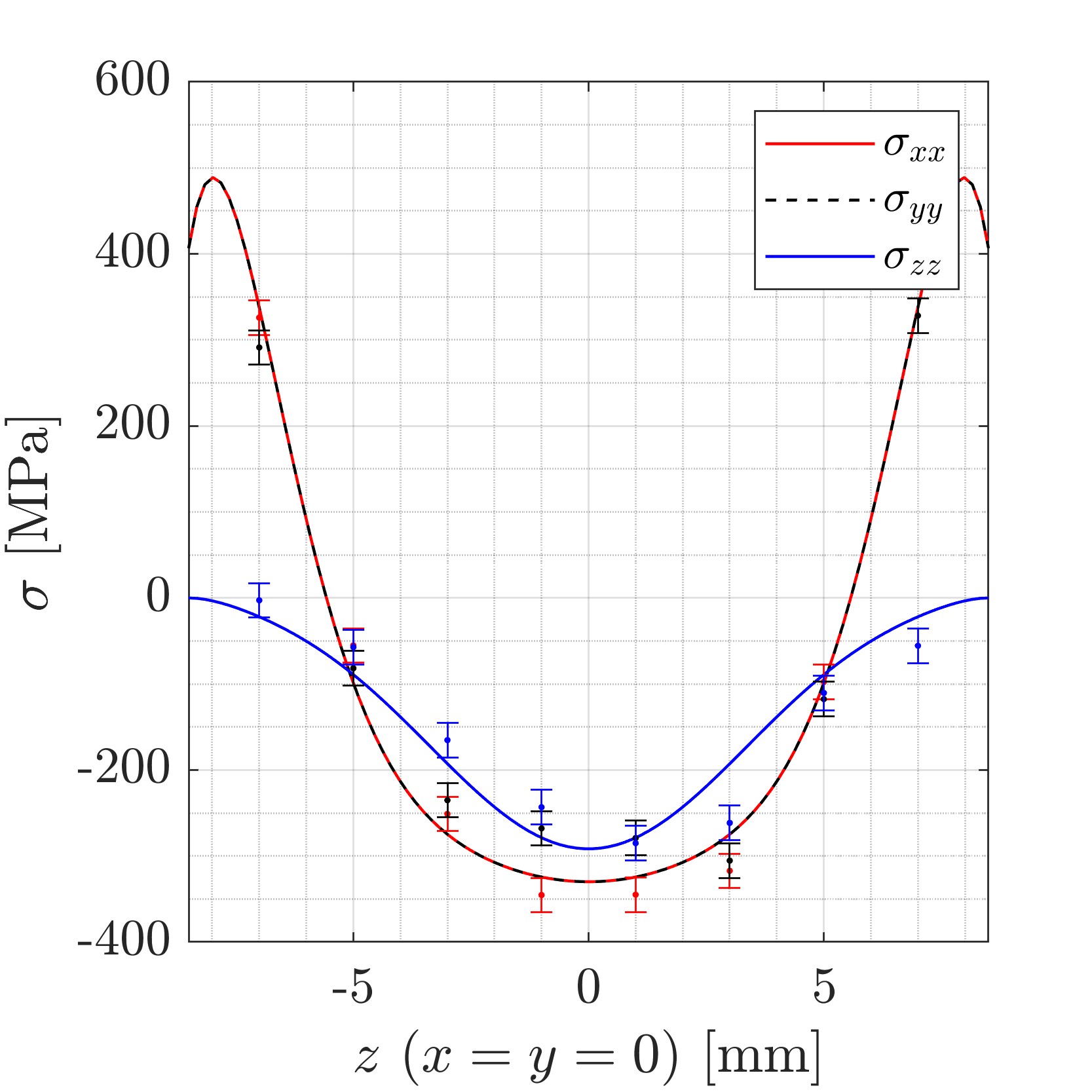}
        \put(-172,155){(d)}
    	\caption{Residual stress in the $17\times 17 \times 17$ mm additively manufactured Inconel cube shown in Figure \ref{SLMCubes}a. (a) Experimental data from the KOWARI instrument measured over a regular $8\times 8$ grid with its extremities inset from the boundary by 1 mm (from \cite{wensrich2024residual}).   The colourmap shown is an interpolation of this measured data without extrapolation (the black boarder represents the sample boundaries). (b) A fitted stress distribution based on a polynomial Maxwell potential of order $3 \times 8$ in $x$, $y$ and $z$. (c) and (d) Profiles of residual stress along the $y$ and $z$ axes respectively; experimental measurements shown as error bars, the fitted Maxwell distribution shown as solid lines.  Note that symmetry dictates that shear components should be zero along these lines.}
        \label{SLMCubeStress}
\end{figure}

As discussed in Section \ref{Sec:Range-NullDecomp}, this fitted stress field directly defines a trivial solution to the inverse eigenstrain problem of the form $\epsilon^*=\epsilon^{*\perp_C}=-S:\sigma$.  This trivial solution is shown in Figure \ref{SLMEigenstrain}a.

\begin{figure}[!htb]
    	\centering
        \includegraphics[width=0.45\linewidth]{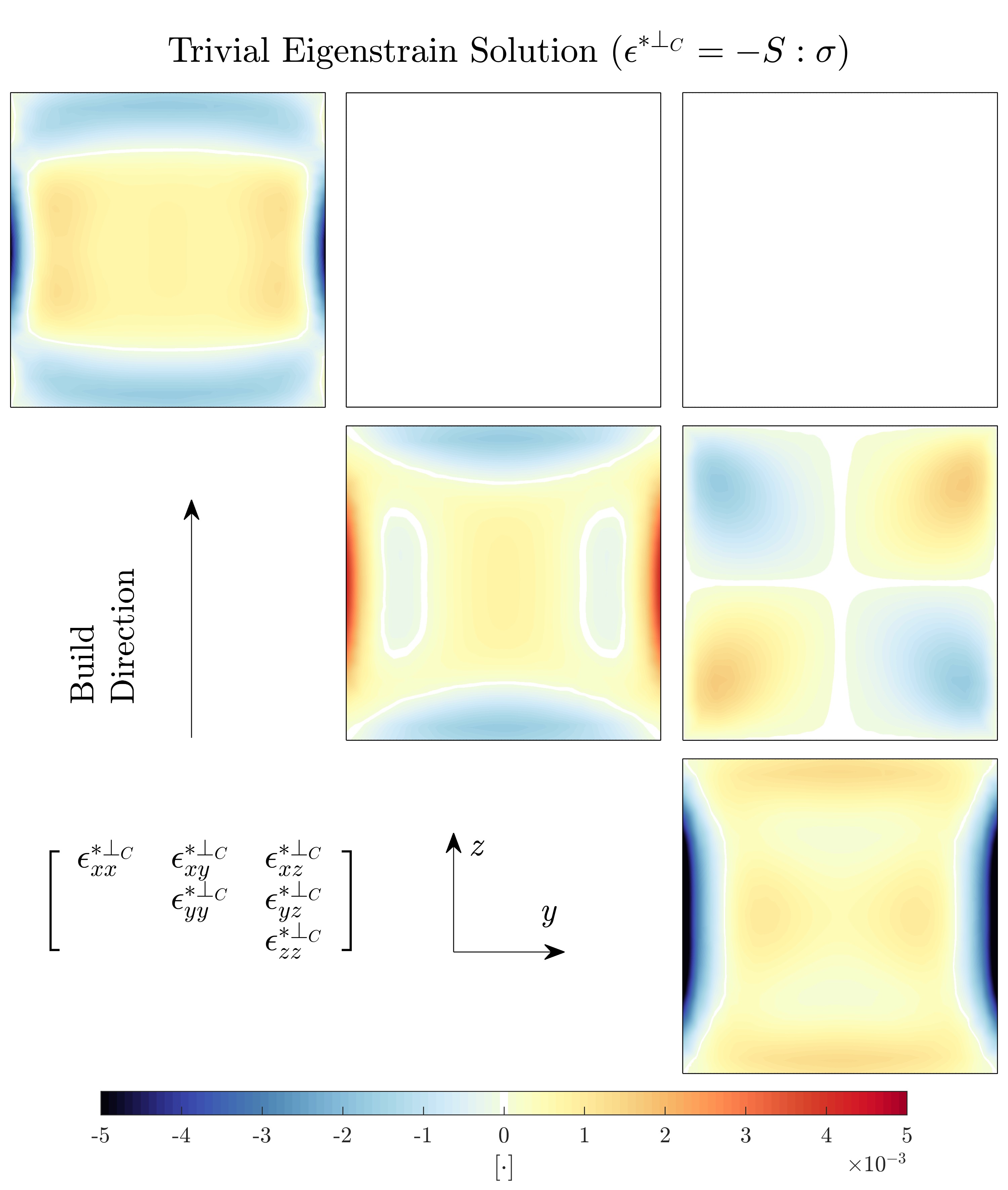}
        \put(-170,195){(a)}
        \includegraphics[width=0.45\linewidth]{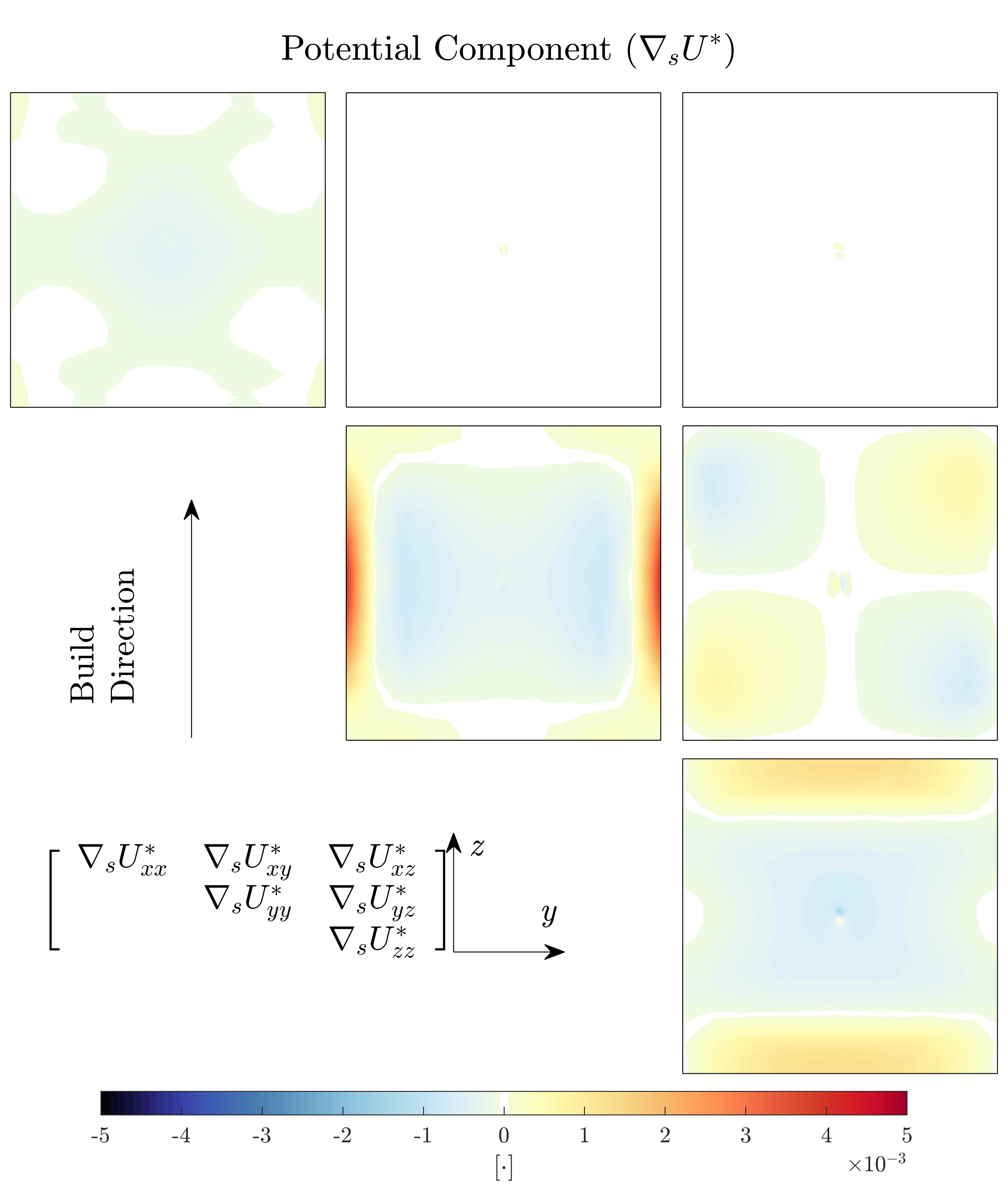}
        \put(-170,195){(b)}\\
        \includegraphics[width=0.45\linewidth]{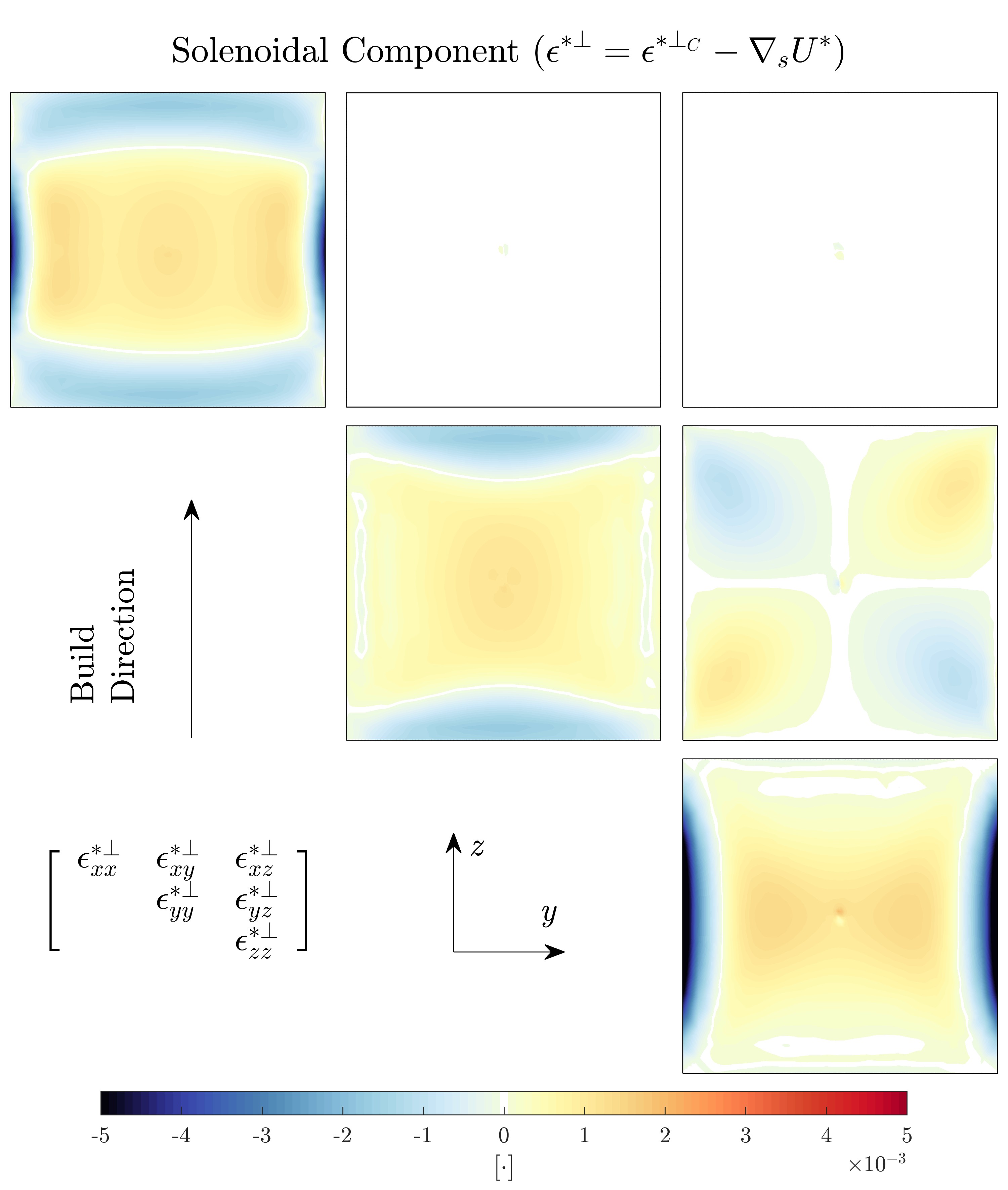}
        \put(-170,195){(c)}
        \includegraphics[width=0.45\linewidth]{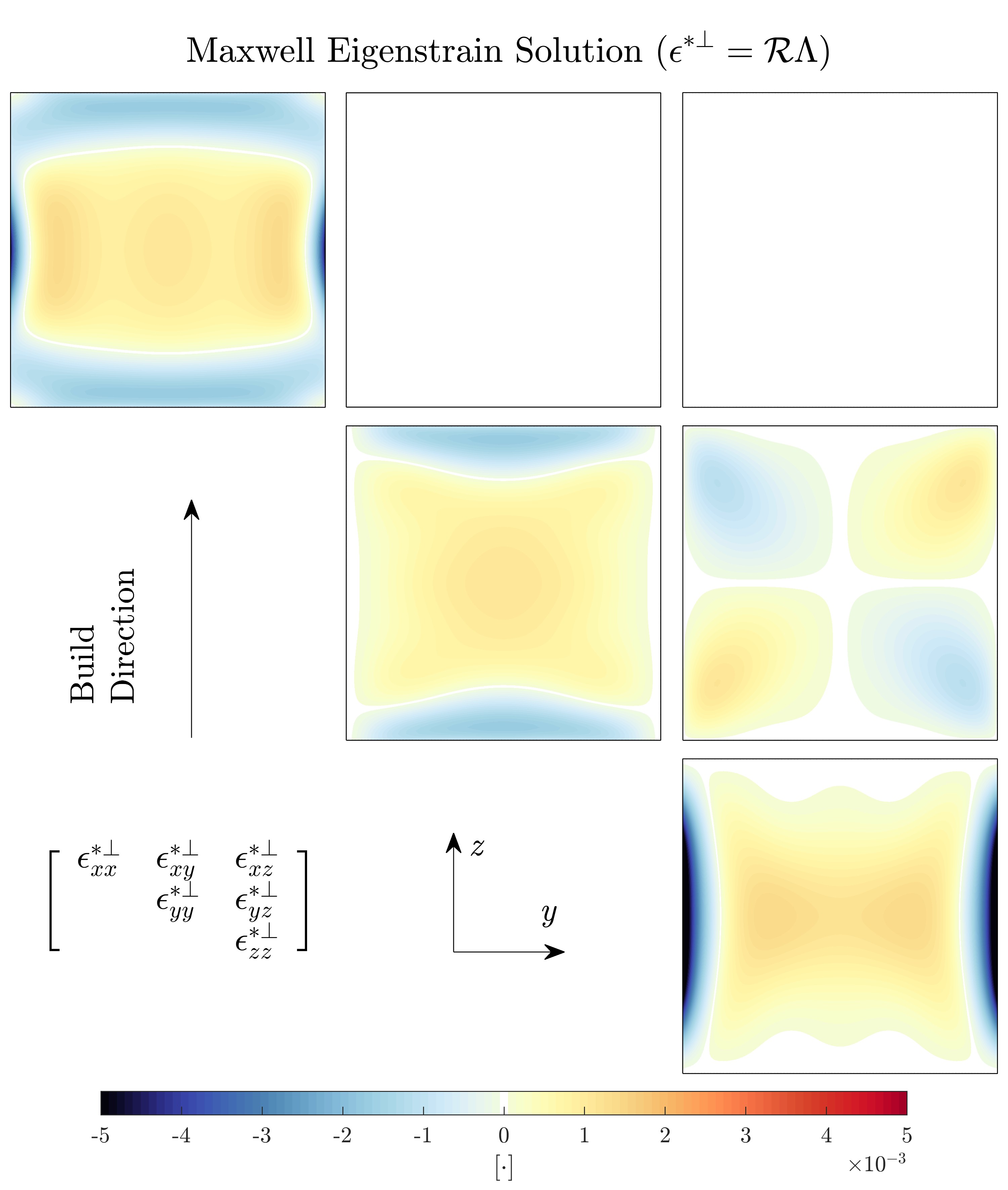}
        \put(-170,195){(d)}
    	\caption{Computed eigenstrain corresponding to the measured stress within the SLM cube.  In each case, the distribution over the central cross section is shown.  (a) The trival eigenstrain solution $\epsilon^{*\perp_C}=-\epsilon$ calculated from the fitted Maxwell stress distribution shown in Figure \ref{SLMCubeStress}b using Hooke's law.  (b) The potential component of the trivial eigenstrain $\nabla_sU^*$ calculated as a finite element solution to \eqref{DecompPDE1}.  (c) The corresponding solenoidal component of the trivial eigenstrain solution $\epsilon^{*\perp}=\epsilon^{*\perp_C}-\nabla_sU^*$.  (d) The Maxwell inverse-eigenstrain solution chosen from the null-complement as the least-squares solution corresponding to the observed stress field.}
        \label{SLMEigenstrain}
\end{figure}

\subsection{Inverse eigenstrain model}

Noting that divergence-free and zero boundary flux eigenstrains can also be represented by the same set of polynomials, we can use a similar process to solve the inverse eigenstrain problem with solutions restricted to the null-compliment. This involves a similar process to before, with the exception that the columns of $A$ are now formed by forward-calculated solutions for stress corresponding to each individual basis function, now interpreted as an eigenstrain field.

Practically, this can be achieved using a finite element model within the MATLAB PDE solver toolbox.  Our implementation consisted of 6716 quadratic finite elements generated within the cube boundary using the in-built MATLAB mesh-generation tool.  The right-hand-sides of \eqref{EquilibriumEigenstrain} and \eqref{BoundaryCondition} were implemented as non-constant coefficients and boundary tractions in the PDE solver via anonymous functions.

Figure \ref{SLMEigenstrain}d shows the resulting inverse eigenstrain solution depicted over the central cross section of the cube.  The resulting stress distribution is not shown, but is effectively identical to that shown in Figure \ref{SLMCubeStress}b.  While this Maxwell eigenstrain solution is similar in magnitude to the trivial eigenstrain (i.e. Figure \ref{SLMEigenstrain}a), the two distributions are clearly not the same.  This is particularly noticeable around the $y=\pm L$ and $z=\pm L$ boundaries.  The difference between these two solutions is in the form of a potential which can be found by decomposing the trivial solution as follows;

From the decomposition \eqref{HelmholtzDecomp}, we can write
\[
\epsilon^{*\perp_C}=\nabla_sU^*+\epsilon^{*\perp},
\]
where $\epsilon^{*\perp}$ is divergence-free with no boundary flux.  Substituting into \eqref{DecompPDE}, we obtain a boundary value problem for $\nabla_s U^*$ of the form
\begin{equation}
\label{DecompPDE1}
\text{Div }(\nabla_s U^*)=\text{Div }(\epsilon^{*\perp_C}),
\text{ where } \hspace{1ex}
\nabla_s U^*\cdot n\Big|_{\partial\Omega}=\epsilon^{*\perp_C}\cdot n\Big|_{\partial\Omega}.
\end{equation}

This can be solved numerically in much the same way as the forward eigenstrain problem using the same finite element model as before (with $C$ replaced by the identity).  Figure \ref{SLMEigenstrain}b shows the resulting potential component with the corresponding solenoidal component shown in Figure \ref{SLMEigenstrain}c.  Apart from a small numerical artifact at the origin, this solenoidal component is essentially identical to the Maxwell inverse eigenstrain solution (i.e. Figure \ref{SLMEigenstrain}d; the inverse eigenstrain solution from within the null-compliment).

So from this perspective, what can be said about the inverse eigenstrain approach?  In some sense it provides a framework through which we can model stress fields, but on the other hand, the existence of the trivial solution implies that no new information is gained by solving the problem.  This may not be the case if additional information is available; e.g perhaps there are circumstances where the eigenstrain tensor is known to be hydrostatic, uniaxial, principal in certain directions, or some other prescribed form.  In these cases, perhaps the inverse problem can shed light on magnitudes, but it is unlikely that it can test the underlying assumptions.  Fundamentally, the fact that an eigenstrain solution fits an observed stress field does not imply it is the physically reality.

\section{A brief note on strain tomography based on the Longitudinal Ray Transform}

Bragg-edge tomography is an experimental technique for imaging elastic strain based on energy-resolved neutron transmission measurement.  A full description can be found elsewhere (e.g. \cite{Tomo:abbey2012neutron,Tomo:busi2022bragg,Tomo:gregg2017tomographic,Tomo:gregg2018tomographic,Tomo:hendriks2017bragg,Tomo:hendriks2019tomographic,Tomo:hendriks2019implementation,Tomo:kirkwood2015neutron,Tomo:lionheart2015diffraction,Tomo:wensrich2024direct,Tomo:wensrich2024general,Tomo:zhu2023bragg}), but briefly, the technique is focused on tomographic reconstruction of an elastic strain field from a set of multiple lower-dimensional projections. 

\begin{figure}
    \centering
    \includegraphics[width=0.5\linewidth]{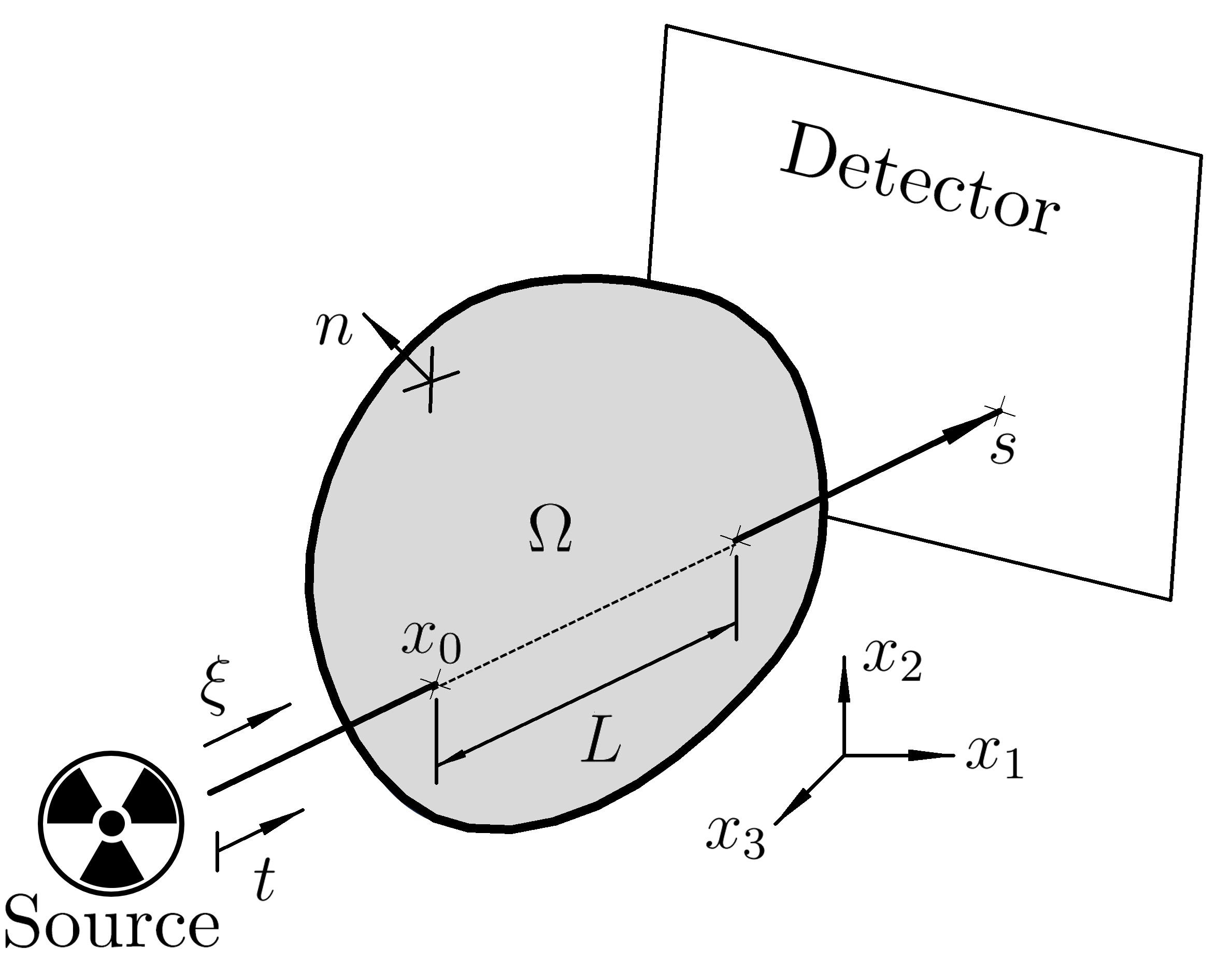}
    \caption{Geometry of the Longitudinal Ray Transform from the perspective of strain tomography on the bounded object $\Omega$.}
    \label{fig:LRTGeom}
\end{figure}

In three-dimensions (see Figure \ref{fig:LRTGeom}) these projections are in the form of two-dimensional images where the value at each pixel is defined by the Longitudinal Ray Transform (LRT) of the elastic strain field\footnote{Note that experimental measurements are usually in the form of averages along the ray path through the interior of the sample.  i.e.
\[
\bar\epsilon(s,\xi)=\frac{1}{L}\int_{-\infty}^{\infty} \epsilon_{ij}(s+t\xi)\xi_i\xi_j \mathrm{d}t.
\]
From a known geometry, it is a simple process to convert these averages to the LRT.
}
\[
I\epsilon(s,\xi)=\int_{-\infty}^{\infty} \epsilon_{ij}(s+t\xi)\xi_i\xi_j \mathrm{d}t.
\]
The associated strain tomography problem focuses on the reconstruction of $\epsilon$ from measurements of $I\epsilon$ (i.e. the inversion of the LRT).  However, this task is hindered by a significant issue in the form of a large null-space for $I$ consisting of all potential fields of the form $\nabla_sU$ with $U\big|_{\partial\Omega}=0$ \cite{sharafutdinov2012integral}.  Any field of this form cannot be observed by the LRT.  

As with our eigenstrain problem, this forms a natural orthogonal decomposition for elastic strain of the form
\[
\epsilon=\nabla_sU+{^s\epsilon}
\]
where $^s\epsilon$ is a divergence-free solenoidal component corresponding to a potential component satisfying $U\big|_{\partial\Omega}=0$.  Courtesy of Sharafutdinov \cite{sharafutdinov2012integral}, an explicit reconstruction formula exists for $^s\epsilon\in L^2(\mathcal{S}^2;\mathbb{R}^3)$ of the form \cite{Tomo:wensrich2024general}
\begin{equation}
{^s}\epsilon = \frac{1}{4\pi^2}(- \Delta)^{1/2}\Big[4 - (\text{\bf{I}}-\Delta^{-1}\nabla_s^2) \textnormal{tr} \Big] \mathcal{X}^* \xi\otimes\xi I\epsilon,
\label{3DInvFormula}
\end{equation}
where $\text{\bf{I}}$ is the rank-2 identity, $\textnormal{tr}$ is the trace operator, $\Delta$ is the Laplacian operator with fractional and negative powers implemented through the Fourier domain and $\mathcal{X}^*$ represents the standard back-projection operator as used in scalar x-ray tomography
\[
\mathcal{X}^*g(x)=\int_{\mathbb{S}^2} g(x,\xi) \mrd\xi.
\]

Armed with our earlier analysis of the eigenstrain problem, we can now recognise this reconstructed component in a new light.  Given that $-\epsilon$ is a trivial solution to the inverse eigenstrain problem and $\nabla_sU$ is in the associated null space, $-^s\epsilon$ must also be a valid solution to the eigenstrain problem.  In other words, the stress field within $\Omega$ can be reconstructed by solving the forward eigenstrain problem for an eigenstrain of $\epsilon^*=-{^s\epsilon}$.

\section{Conclusions}

The eigenstrain concept is central to the analysis of residual stress in real physical systems.  In a practical sense, real physical objects are routinely subject to eigenstrains in the form of plastic deformation, thermal expansion, and/or dilation/shrinkage due to phase change, etc.  From this perspective, the forward problem is fundamental in terms of modeling the behaviour of real systems and mechanical processes.

The practicality of the inverse problem is more nuanced.  The inherently ill-posed nature of this problem means that a great deal of care should be exercised when drawing conclusions from reconstructed eigenstrain fields.  Without detailed prior knowledge, in a general sense the best we can hope for is a reconstruction that differs from reality by some unknown potential field that can never be observed.  The simplest of these reconstructions is trivially related to the known stress field, and from this perspective it is clear that a general solution to the inverse problem does not generate any new information.

That is not to say that the inverse problem has no practical utility.  As demonstrated by our example involving ancient Roman medical tools, equilibrium imposes a strong constraint that can be applied through the eigenstrain framework to potentially uncover other information such as the $d_0$ distribution within a sample.  In some circumstances, this may be the only non-destructive means available.  There are also cases where prior knowledge may exist regarding the form of the eigenstrain within a sample.  In these scenarios, the inverse problem may be able to inform on reality, however this is far from a general statement.

Outside of these direct applications, we also observed that there is a great deal of similarity between the null space of the eigenstrain problem and the null space of the LRT.  This similarity has allowed us to draw a link between the solenoidal component of elastic strain recovered by LRT inversion formulas due to Sharafutdinov \cite{sharafutdinov2012integral} and inverse eigenstrain solutions.  Through this link, we have provided a potential avenue for the reconstruction of stress within physical objects from the LRT of elastic strain.

\section{Acknowledgements}

This work is supported by the Australian Research Council through a Discovery Project Grant (DP170102324). Access to the KOWARI diffractometer was made possible through an ANSTO Program Proposal PP6050 and regular beamtime proposals P14244 and P16860.  Additional support from AINSE Limited was also provided during the experimental work.  

The authors wish to acknowledge the assistance of the RD Milns Antiquities Museum at the University of Queensland in this work; particularly for the loan of the Roman medical tools.

We also acknowledge the Additive Manufacturing Research Laboratory (AMRL) at RISE IVF in Sweden for manufacturing the SLM specimens.

CM Wensrich, S Holman and WBL Lionheart would like to thank the Isaac Newton Institute for Mathematical Sciences at the University of Cambridge for support and hospitality during the program `Rich and Non-linear Tomography: A Multidisciplinary Approach'.  Foundational work on this paper was undertaken as a part of this program. The RNT program was supported by EPSRC grant number EP/R014604/1.  

While in Cambridge, CM Wensrich received support from the Simons Foundation and would also like to thank Clare Hall for their support and hospitality over this period.

\bibliography{references}

\end{document}